\documentclass[twoside,10pt]{article}
\usepackage{amsmath,latexsym,amsfonts,indentfirst,amsthm,amsxtra,amssymb,bm}

\newtheorem{lemma}{Lemma}[section]
\newtheorem{theorem}{Theorem}[section]

\newtheorem{remark}{Remark}[section]

\numberwithin{equation}{section}

\begin{document}

\title{\bf Global Existence and Decay of Solutions to the Fokker-Planck-Boltzmann Equation}

\author{
{\bf Linjie Xiong}\thanks{Corresponding author.
E-mail:  xlj@whu.edu.cn; tao.wang@whu.edu.cn; mathwls08@gmail.com}, \quad {\bf Tao Wang},\quad and \quad {\bf Lusheng Wang}\\[2mm]
School of Mathematics and Statistics\\
Wuhan University, Wuhan 430072, China}
\date{}

\vskip 0.2cm

\maketitle

\begin{abstract}
  The Cauchy problem to the Fokker-Planck-Boltzmann equation
  under Grad's angular cut-off assumption is investigated. When the initial data is a small perturbation of
  an equilibrium state, global existence and optimal temporal decay estimates of classical solutions are established. Our analysis is based on the coercivity of the Fokker-Planck operator and an elementary weighted energy method.
\end{abstract}

\section{Introduction and Main Results}
The Fokker-Planck-Boltzmann equation models the motion of particles in a thermal bath where the bilinear interaction is one of the main characters \cite{Bisi, Cercignani88, Loyalka}.
Mathematically, the Fokker-Planck-Boltzmann equation takes the following form:
\begin{equation} \label{FPB}
  \partial_t f+\xi\cdot\nabla_x f=Q(f,f)+{\epsilon}\nabla_{\xi}\cdot(\xi f) +\kappa\Delta_{\xi} f,
\end{equation}
where the nonnegative unknown function $f=f(t,x,\xi)$ represents the density of particles
at position $x \in\mathbb{R}^3$ and time $t\geq 0$ with velocity $\xi \in\mathbb{R}^3 $
and $\epsilon, \kappa$ are given nonnegative constants.
The collision operator $Q$ is a bilinear operator which acts only on the velocity variables $\xi$ and is local in $(t, x)$ as
\begin{equation} \label{Q}
  Q(f,g)(\xi)
  =\int_{\mathbb{R}^3\times S^2}q(\xi-\xi_{*},\omega)\left\{f(\xi_{*}')g(\xi')-f(\xi_*)g(\xi)\right\}
  d\omega d\xi_{*}.
\end{equation}
Here $\xi$, $\xi_{*}$ and $\xi'$, $\xi'_{*}$ are the velocities of a pair of particles before and after collision.
we assume these collisions to be elastic so that
\begin{equation*}
  \xi'=\xi-[(\xi-\xi_{*})\cdot\omega]\omega,\quad \xi_{*}'=\xi_{*}+[(\xi-\xi_{*})\cdot\omega]\omega,
  \quad\omega\in S^2.
\end{equation*}

The Boltzmann collision kernel $q(\xi-\xi_{*},\omega)$ for a monatomic gas is, on physical
grounds, a non-negative function which only depends on the relative velocity $|\xi-\xi_{*}|$
and on the angle $\theta$ through
$  \cos \theta=\omega\cdot(\xi-\xi_{*})/|\xi-\xi_{*}|.$
There are two important model cases in physics:
\begin{itemize}
  \item[$\bullet$] Hard spheres, i.e., particles which collide bounce on each other like billiard
                   balls. In this case
                   \begin{equation*}
                     q(|\xi-\xi_*|,\omega)=|(\xi-\xi_*)\cdot\omega|=|\xi-\xi_*||\cos \theta|.
                   \end{equation*}
  \item[$\bullet$] Inverse-power law potentials, i.e., particles which interact according to
                   a spherical intermolecular repulsive potential of the form
                   \[\phi(r)=r^{-(s-1)},\quad s\in(2,\infty),\]
                   then one can show that
                   \begin{equation*}
                     q(|\xi-\xi_*|,\omega)=|\xi-\xi_*|^{\gamma}B(\theta),\quad  \gamma=1-\frac{4}{s-1}.
                   \end{equation*}
                   As for the function $B$, it is only implicitly defined, locally smooth, and has a non-integrable singularity
                   $$ B(\theta)=|\cos \theta|^{-\gamma'}q_0(\theta),\quad \gamma'=1+\frac{2}{s-1},$$
                   where $q_0(\theta)$ is bounded, $q_0(\theta)\neq 0$ near $\theta=\pi/2.$
\end{itemize}

We consider the Cauchy problem of (\ref{FPB}) with prescribed initial data
\begin{equation}  \label{f_0}
  f(0,x,\xi)=f_0(x,\xi).
\end{equation}
Throughout this manuscript, we assume that
${{\epsilon}=\kappa}>0$ such that the global Maxwellian
$M=(2\pi)^{-3/2}e^{-|\xi|^2/2}$ is an equilibrium state of (\ref{FPB})
and the collision kernels satisfy Grad's angular cut-off assumption:
\begin{equation}  \label{cutoff}
  q(|\xi-\xi_*|,\omega)=|\xi-\xi_*|^{\gamma}B(\theta),\quad 0\leq B(\theta)\leq C|\cos \theta|,
  \quad -3<\gamma\leq 1.
\end{equation}

Our goal in this paper is to obtain the global existence and optimal temporal decay estimates
of classical solutions for (\ref{FPB}) and (\ref{f_0}) with ${\epsilon}=\kappa>0$ when the initial data $f_0$ is near the global Maxwellian
$
M=(2\pi)^{-3/2}e^{-|\xi|^2/2}.
$
To this end, if we use $u$ to denote the perturbation of  $f$ around the Maxwellian $M$ as
\begin{equation*}
  f=M+M^{1/2}u,
\end{equation*}
then the Cauchy problem (\ref{FPB}) and (\ref{f_0}) can be reformulated as
\begin{align}
\label{u}
  \partial_t u+\xi\cdot\nabla_x u&=Lu +\Gamma(u,u)+{\epsilon}L_{FP}u,\\[2mm]
\label{u_0}
  u(0,x,\xi)&=u_0(x,\xi)=M^{-1/2}(f_0-M).
\end{align}
Here, the linear operator $L$, the bilinear form $\Gamma(u_1,u_2)$ and
the classical linearized Fokker-Planck operator $L_{FP}$ are, respectively, given by
\begin{equation*}
  \begin{aligned}
    Lu&=M^{-\frac{1}{2}}\left\{Q(M,M^{1/2}u)+Q(M^{1/2}u,M)\right\},\\[2mm]
    \Gamma(u_1,u_2)&=M^{-\frac{1}{2}}Q(M^{1/2}u_1,M^{1/2}u_2),\\[2mm]
    L_{FP}u&=\Delta_{\xi}u+\frac{1}{4}(6-|\xi|^2)u.
  \end{aligned}
\end{equation*}
It is well known that for the linearized collision operator $L$, one has
\begin{equation*}
  Lg(\xi) = -\nu(\xi) g(\xi) + Kg(\xi),
\end{equation*}
where the collision frequency is
\begin{equation*}
  \nu(\xi)=\int_{\mathbb{R}^3\times S^2}|\xi-\xi_{*}|^{\gamma}q_0(\theta)M(\xi_{*})d\omega d\xi_{*}
  \sim(1 + |\xi|)^{\gamma},
\end{equation*}
and the operator $K$ is defined by
\begin{equation*}
  \begin{aligned}
    Ku(\xi)=&\int_{\mathbb{R}^3\times S^2}
             |\xi-\xi_{*}|^{\gamma}q_0(\theta)M^{1/2}(\xi_{*})M^{1/2}(\xi_{*}')u(\xi')d\omega d\xi_{*}\\
            &+\int_{\mathbb{R}^3\times S^2}
             |\xi-\xi_{*}|^{\gamma}q_0(\theta)M^{1/2}(\xi_{*})M^{1/2}(\xi')u(\xi_{*}')d\omega d\xi_{*}\\
            &-\int_{\mathbb{R}^3\times S^2}
             |\xi-\xi_{*}|^{\gamma}q_0(\theta)M^{1/2}(\xi_{*})M^{1/2}(\xi)u(\xi_{*})d\omega d\xi_{*}.
  \end{aligned}
\end{equation*}
Furthermore, the operator $L$ is non-positive, the null space of $L$ is the five dimensional space
\begin{equation*}
  \mathcal{N}=\textrm{span}\left\{M^{1/2},\xi_jM^{1/2}(j=1,2,3),|\xi|^2M^{1/2}\right\},
\end{equation*}
and $-L$ is locally coercive in the sense that there is a positive constant $\lambda_0$ such that
(see \cite{Cercignani}, \cite{Guo03}, \cite{Mouhot})
\begin{equation} \label{L0}
  -\int_{\mathbb{R}^3}uLud\xi\geq \lambda_0\int_{\mathbb{R}^3}\nu(\xi)|\{{\bf I}-{\bf P}\}u|^2d\xi
\end{equation}
holds for $u = u(\xi)$, where {\bf I} means the identity operator and {\bf P} denotes
its $\xi$-projection from $L^2_{\xi}(\mathbb{R}^3)$ onto the null space $\mathcal{N}$.
As in \cite{Guo04}, for any function $u(t, x, \xi)$, we can write {\bf P} as
\begin{equation*}
  \begin{cases}
    {\bf P}u=\{a(t,x)+b(t,x)\cdot\xi+c(t,x)(|\xi|^2-3)\}M^{1/2},\\[2mm]
    a=\int_{\mathbb{R}^3}M^{1/2}ud\xi,
    \quad
    b=\int_{\mathbb{R}^3}\xi M^{1/2}u d\xi,\\[2mm]
    c=\frac{1}{6}\int_{\mathbb{R}^3}(|\xi|^2-3)M^{1/2}ud\xi,
  \end{cases}
\end{equation*}
so that we have the macro-micro decomposition introduced in \cite{Guo04}
\begin{equation} \label{macro-micro}
  u(t,x,\xi)={\bf P}u(t,x,\xi)+\{{\bf I}-{\bf P}\}u(t,x,\xi).
\end{equation}
Here, ${\bf P}u$ and $\{{\bf I}-{\bf P}\}u$ is called the macroscopic component and the microscopic
component of $u(t, x, \xi)$, respectively. For later use, one can rewrite {\bf P} as
\begin{equation*}
  \begin{cases}
    {\bf P}u={\bf P}_0u\oplus{\bf P}_1u,\\[2mm]
    {\bf P}_0u=a(t,x)M^{1/2},\\[2mm]
    {\bf P}_1u=\left\{b(t,x)\cdot\xi+c(t,x)(|\xi|^2-3)\right\}M^{1/2}.
  \end{cases}
\end{equation*}

\textbf{Notations.} Throughout this paper, $C$ denotes some positive (generally large)
constant and $\lambda$ denotes some positive (generally small) constant, where both $C$ and
$\lambda$ may take different values in different places.
$A \lesssim B$ means there exists a constant $C>0$ such that $A\leq  CB$ holds uniformly.
$A \sim B$ means $A\lesssim B$ and $B\lesssim A$.
For the multi-indices $\alpha=(\alpha_1,\alpha_2,\alpha_3)$ and $\beta=(\beta_1,\beta_2,\beta_3)$,
$\partial_{\beta}^{\alpha}=\partial_{x_1}^{\alpha_1}\partial_{x_2}^{\alpha_2}\partial_{x_3}^{\alpha_3}
\partial_{\xi_1}^{\beta_1}\partial_{\xi_2}^{\beta_2}\partial_{\xi_3}^{\beta_3}.$
Similarly, the notation $\partial^{\alpha}$ will be used when $\beta=0$, and likewise for $\partial_{\beta}$.
The length of $\alpha$ is denoted by $|\alpha|=\alpha_1+\alpha_2+\alpha_3$.
$\beta\leq\alpha$ means that $\beta_j\leq\alpha_j$ for each $j=1,2,3$,
and $\alpha<\beta$ means that $\beta\leq\alpha$ and $|\beta|<|\alpha|$.
For notational simplicity,
let $\langle\cdot,\cdot\rangle$ denote the $L^2$ inner product in $\mathbb{R}^3_{\xi}$
with the $L^2$ norm $|\cdot|_2$,
and let $(\cdot, \cdot)$ denote the $L^2$ inner product either in $\mathbb{R}^3_{x}\times\mathbb{R}^3_{\xi}$
or in $\mathbb{R}^3_{x}$ with the $L^2$ norm $\|\cdot\|$.
Moreover, we define
$$
|g|_{\nu}^2=\langle \nu(\xi) g,g\rangle, \quad
\|g\|_{\nu}^2=( \nu(\xi) g,g).
$$
For an integer $m \geq 0$, we use $H^m$ to denote the usual Sobolev space.
We also define the space
$Z_q=L^2(\mathbb{R}_{\xi}^3; L^q(\mathbb{R}_{x}^3))$
for $q \geq 1$ with the norm
\begin{equation*}
  \|u\|_{Z_q}=\left(\int_{\mathbb{R}^3}\left(\int_{\mathbb{R}^3}|u(x,\xi)|^{q}dx\right)^{2/q}d\xi\right)^{1/2},
  \quad u=u(x,\xi)\in Z_q.
\end{equation*}
For an integrable function $g: \mathbb{R}^3\rightarrow\mathbb{R}$, its Fourier transform $\widehat{g}=\mathcal{F}g$ is defined by
\begin{equation*}
  \widehat{g}(k)=\mathcal{F}g(k)=\int_{\mathbb{R}^3}e^{-2\pi ix\cdot k}g(x)dx,\quad x\cdot k=\sum_{j}x_jk_j.
\end{equation*}
for $k\in \mathbb{R}^3$, where $i=\sqrt{-1}\in \mathbb{C}$ is the imaginary unit.
For two complex vectors $a,b\in \mathbb{C}^3,(a|b)=a\cdot \overline{b}$
denotes the dot product over the complex filed, where $\overline{b}$ is the complex conjugate of $b$.

For $q\in\mathbb{R}$, the velocity weight function $w_{q}=w_{q}(\xi)$ is always denoted by
\begin{equation}\label{weight}
  w_{q}(\xi)=\langle\xi\rangle^{q-\gamma}
\end{equation}
with $\langle\xi\rangle=(1+|\xi|^2)^{1/2}.$
For an integer $N$ and $l\geq N$, we define the instant energy functional
\begin{equation} \label{E}
  \mathcal{E}_{q,l}(u)(t)\equiv\sum_{|\alpha|+|\beta|\leq N}\|w_{q}^{l-|\beta|}\partial_{\beta}^{\alpha}u(t)\|^2,
\end{equation}
and the dissipation rate
\begin{equation} \label{D}
\begin{aligned}
  \mathcal{D}_{q,l}(u)(t)\equiv&\sum_{1\leq|\alpha|\leq N}\left\|\partial_x^{\alpha}{\bf P}u(t)\right\|^2
                            +{\epsilon}\sum_{|\alpha|\leq N}\left\|\{{\bf I}-{\bf P}_0\}\partial_x^{\alpha} u\right\|^2\\[2mm]
                           &+\sum_{|\alpha|+|\beta|\leq N}\|w_{q}^{l-|\beta|}
                            \partial_{\beta}^{\alpha}\{{\bf I-P}\}u(t)\|_{\nu}^2.
\end{aligned}
\end{equation}
We remark that our energy functional and dissipation rate which are not necessary to include the temporal derivatives which are different from \cite{Zhong_Li2}.
The main result of this paper is stated as follows: For the hard potential case, we have
\begin{theorem} \label{theorem1}
  Let $0\leq \gamma\leq 1$, $l\geq N\geq 4$, and $q\geq 1$.
  Assume that Grad's angular cut-off (\ref{cutoff}) is satisfied and that $f_0(x,\xi)=M+M^{1/2}u_0(x,\xi)\geq 0$. Then we have
  \begin{itemize}
  \item[{\rm (i)}] If there exists a sufficiently small $\delta_0>0$ such that $\mathcal{E}_{q,l}(u_0)\leq \delta_0$ and
  $(q-\gamma)^2{\epsilon}\leq \delta_0$, the Cauchy problem (\ref{u})-(\ref{u_0}) admits a unique global solution $u$ which satisfies
  $f(t,x,\xi)=M+M^{1/2}u(t,x,\xi)\geq 0$ for every $t\geq 0$;
  \item[{\rm (ii)}] If we assume further that $\gamma\leq 2 l(q-\gamma)$ and that there exists a sufficiently small positive constant $\delta_1>0$ such that
  $\mathcal{E}_{q,l}(u_0)+\|u_0\|_{Z_1}^2\leq \delta_1$ and $(q-\gamma)^2{\epsilon}\leq \delta_1$, the unique global solution $u(t,x,\xi)$ obtained above
  satisfies the following optimal temporal decay estimates
  \begin{equation*}
     \sup_{t\geq0}\left\{(1+t)^{\frac{3}{2}}\mathcal{E}_{q,l}(u)(t)\right\}\lesssim \delta_1.
  \end{equation*}
  \end{itemize}
\end{theorem}
For the soft potential case, we have
\begin{theorem}\label{theorem2}
  Let $-3<\gamma<0$, $l\geq N\geq 8$, and $q\geq 1$.
  Assume that Grad's angular cut-off (\ref{cutoff}) is satisfied and that $f_0(x,\xi)=M+M^{1/2}u_0(x,\xi)\geq 0$. Then we have
  \begin{itemize}
  \item[{\rm (i)}] If there exists a sufficiently small $\delta_0>0$ such that $\mathcal{E}_{q,l}(u_0)\leq \delta_0$ and
  $(q-\gamma)^2{\epsilon}\leq \delta_0$, the Cauchy problem (\ref{u})-(\ref{u_0}) admits a unique global solution $u(t,x,\xi)$ which satisfies $f(t,x,\xi)=M+M^{1/2}u(t,x,\xi)\geq 0$ for every $t\geq 0$;
  \item[{\rm (ii)}] If we assume further that $l\geq N+1$ and $\gamma(1-l_0)\leq 2(q-\gamma)(l-1)$ for some $l_0>3/2$ and that there exists a sufficiently small $\delta_1>0$ such that
  $\mathcal{E}_{q,l}(u_0)+\|\langle\xi\rangle^{-\gamma l_0/2}u_0\|_{Z_1}^2\leq \delta_1$
  and $(q-\gamma)^2{\epsilon}\leq \delta_1$, the unique global solution $u(t,x,\xi)$ obtained above satisfy the following optimal temporal decay estimate
  \begin{equation*}
     \sup_{t\geq0}\left\{(1+t)^{\frac{3}{2}}\mathcal{E}_{q,l-1}(u)(t)\right\}\lesssim \delta_1.
  \end{equation*}
  \end{itemize}
\end{theorem}

\begin{remark}
  The analysis here can be used to deal with the case when ${\epsilon}={\epsilon}(t)>0$ and similar results can also be obtained provided that $(q-\gamma)^2\epsilon(t)\leq \delta_i$ hold for $i=0,1$ and every $t\geq 0$. This means that for the Fokker-Planck-Boltzmann equation (\ref{FPB}) with  ${\epsilon}\equiv 0$ and $\kappa>0$, i.e.
  $$
   \partial_t f+\xi\cdot\nabla_x f=Q(f,f) +\kappa\Delta_{\xi} f,
  $$
  we can use the scaling used in \cite{Li_Matsumura} to transform the above problem into
  (\ref{FPB}) with ${\epsilon}=\kappa={\kappa(1+3\kappa t)^{-1}}$ and similar results can also be obtained provided that $(q-\gamma)^2\epsilon(t)=(q-\gamma)^2\kappa(1+3\kappa t)^{-1}\leq \delta_i$ hold for $i=0,1$ and every $t\geq 0$.
  It is easy to see that a sufficient condition to guarantee the validity of the above inequalities is that $\kappa>0$ is sufficiently small as imposed in \cite{Li_Matsumura} and it is worth to pointing out that when $\gamma\to 1^-$ and by taking $q=1$, one can see that the assumptions $(q-\gamma)^2\kappa(1+3\kappa t)^{-1}\leq \delta_i$ hold even without the smallness  restriction on $\kappa$. In such a sense, our result generalizes the result obtained in \cite{Li_Matsumura} even for the hard sphere intermolecular interaction.
\end{remark}

\begin{remark}
It is worth to point out that here we use the weight function $w_{q}^{l-|\beta|}$ to capture the term $|\xi||\partial_{\beta}u|$
generated by the $\xi$-derivatives $\partial_{\beta}$ acting on the Fokker-Planck operator in term of the weaker dissipation rate $\|\partial_{\beta}u\|_{\nu}$.
\end{remark}

\begin{remark}
   The rates of convergence are optimal under the corresponding assumptions in the sense that
   they coincide with those rates given in $(\ref{linear-equation})$ at the level of linearization.
\end{remark}

There have been a lot of studies on the Fokker-Planck-Boltzmann equation (\ref{FPB}).
\textsc{DiPerna} and \textsc{Lions} \cite{DiPerna_Lions} proved the global existence of
the renormalized solutions for the Cauchy problem (\ref{FPB}) and (\ref{f_0}).
\textsc{Hamdache} \cite{Hamdache} obtained the global existence near the vacuum state in terms of a
direct construction.
It is shown in \cite{Li_Matsumura}
that a strong solution of the equation (\ref{FPB}) for initial data near the global Maxwellian
exists globally in time and tends asymptotically to another time-dependent
self-similar Maxwellian in the large-time limit for the hard sphere case (\ref{cutoff}) with $\gamma=1$.
\textsc{Li} and \textsc{Matsumura} in \cite{Li_Matsumura} first introduced an appropriate scaling to
transform (\ref{FPB}) with ${\epsilon}\equiv 0$ and $\kappa>0$ into
(\ref{FPB}) with ${{{\epsilon}=\kappa}}\to{\kappa(1+3\kappa t)^{-1}}$ and then achieved their goals
by employing the pioneering $L^2$ energy method based on macro-micro decomposition around a local Maxwellian
developed for the Boltzmann equation \cite{Liu_Yang_Yu}, \cite{Liu_Yu}.
For the case $-1\leq \gamma\leq 1$, the long time behavior to the Cauchy problem of (\ref{FPB}), (\ref{f_0}) is studied
by constructing the compensating functions to this system, while the main goal of this paper is to obtain the global existence of classical solutions for (\ref{FPB}) and
(\ref{f_0}) and the corresponding optimal time decay of the solutions under Grad's angular cut-off assumption
for the whole range of intermolecular interaction $-3<\gamma\leq 1$.

In the perturbation theory of the Boltzmann equation for the global
well-posedness of solutions around global Maxwelians,
the energy method was first developed
independently in \cite{Liu_Yu,Liu_Yang_Yu} and in \cite{Guo02,Guo04}.
We also mention the pioneering work \cite{Ukai} and its recent improvement \cite{Ukai_Yang} which are based on
the spectral analysis and the contraction mapping principle.
We remark that the energy method based on macro-micro decomposition around a local Maxwellian
\cite{Li_Matsumura} for the Fokker-Planck-Boltzmann eqution for the hard sphere case
does not apply to the problem under our consideration with $-3<\gamma<1$.
Our approach is based on the methods in \cite{Duan_Yang_Zhao1,Duan_Yang_Zhao2} for the
Vlasov-Poisson-Boltzmann system.
For more information related to the Boltzmann equation and the kinetic theory,
the reader can also refer to \cite{Cercignani,Cercignani88,Glassey,villani02} and references therein.

Before concluding this section, we sketch main ideas used in deducing our results.
One of the main difficulties lies in the fact that the dissipation of the linearized Boltzmann operator $L$ for non hard-sphere potentials can not control the full nonlinear dynamics due to the velocity growth effect of $|\xi||\partial_{\beta}u|$ generated by the $\xi$-derivatives $\partial_{\beta}$ acting on the Fokker-Planck operator. A suitable application of a weight function $w_{q}^{l-|\beta|}$ can indeed yield a satisfactory global existence of classical solution to the Fokker-Planck-Boltzmann equation for the case $-2\leq \gamma\leq 1$, while for the very soft potential case $-3<\gamma<-2$, we cannot close our energy estimate by only employing the coercivity of the linearized collision $L$ as for the case of  $-2\leq\gamma\leq 1$. Still and all, we can combine both the coercivity of $L$ and $L_{FP}$ and divide the integral domain about $\xi$ into two parts: the first part $\{\xi|\langle\xi\rangle\leq R\}$ can be control by the coercivity of $L$ with the smallness of ${\epsilon}$ while the second part $\{\xi|\langle\xi\rangle> R\}$ by the coercivity of $L_{FP}$ when we choose $R$ large enough.

The time rate of convergence to equilibrium is an important topic in the mathematical theory of the physical world. As pointed out in $\cite{Villani_MAMS}$, the exist general structures in which the interaction between a conservative part and a degenerate dissipative part lead to the convergence to equilibrium, where this property was called hypocoercivity. Here, indeed, we provide a concrete example of hypocoercivity property for the nonlinear Fokker-Planck-Boltzmann equation in the framework of perturbation.
We employ the methods developing by Duan and Strain $\cite{Duan_Strain11,Duan_CPAM}$.
For the proof, in the linearized case with a given non-homogeneous source, Fourier analysis
is employed to obtain time-decay properties of the solution operator.
In the nonlinear case energy estimates with the help of the proper Lyaponov-type inequalities
lead to the optimal time-decay rate of perturbed solution under some conditions on initial data.
As in $\cite{Duan_Yang_Zhao2}$, unlike the periodic domain $\cite{Strain_Guo}$,
the main difficult of the deducing the decay rates of solution for the soft potential
is caused by the lack of spectral gap for the linearized collision operator $L$.
We need a more delicate estimate on the time decay of solution to the corresponding linearized equation
in the case of the whole space ${\mathbb{R}^3}$ based on the weighted energy estimates,
a time-frequency analysis method, and the construction of some interactive energy functionals. We also mention that
Zhang and Li \cite{Zhong_Li2} have obtained the similar decay rate for the case $-1\leq\gamma\leq0$ by employing the compensating function which is different from us.

The rest of this paper is arranged as follows.
We prove the global existence of solutions to the perturbed problem
by establishing the a priori energy estimates on the microscopic and macroscopic
dissipations which are derived in Sections 2 and 3, respectively.
In the last section, we devote ourselves to obtaining the optimal temporal decay estimates of
the global solutions for both the hard potentials and the soft potentials.

\section{Macroscopic dissipation}
In this section, we will obtain the macroscopic dissipation rate
$$
\sum_{1\leq|\alpha|\leq N}\|\partial_x^{\alpha}{\bf P}u(t)\|^2
\sim \sum_{|\alpha|\leq N-1}\|\partial^{\alpha}_x\nabla_{x}(a,b,c)(t)\|^2.
$$
To this end, we shall first apply the macro-micro decomposition (\ref{macro-micro})
to the equation (\ref{u}) to discover the macroscopic balance laws satisfied by $(a,b,c)$.
Multiply (\ref{FPB}) by the collision invariants 1, $\xi$ and $|\xi|^2$ to find
the local balance laws
\begin{equation} \label{balance}
  \begin{cases}
    \partial_t\int_{\mathbb{R}^3} fd\xi+\nabla _x\cdot \int_{\mathbb{R}^3} \xi fd\xi=0,\\[2mm]
    \partial_t\int_{\mathbb{R}^3} \xi fd\xi+\nabla _x\int_{\mathbb{R}^3}\xi\otimes\xi fd\xi
    +{\epsilon}\int_{\mathbb{R}^3}\xi fd\xi=0,\\[2mm]
    \partial_t\int_{\mathbb{R}^3} |\xi|^2 fd\xi+\nabla _x\cdot\int_{\mathbb{R}^3}|\xi|^2\xi fd\xi
    +2{\epsilon}\int_{\mathbb{R}^3}(|\xi|^2-3)\xi fd\xi=0.
  \end{cases}
\end{equation}
As in \cite{Duan_Strain11}, define the high-order moment functions $A = (A_{jm})_{3\times3}$
and $B = (B_1, B_2, B_3)$ by
\begin{equation}\label{AB}
  A_{jm}(u)=\left\langle(\xi_j\xi_m-1)M^{1/2},u\right\rangle,\quad
  B_j(u)=\frac{1}{10}\left\langle(|\xi|^2-5)\xi_j M^{1/2},u\right\rangle.
\end{equation}
Plugging $f=M+M^{1/2}{\bf P}u+M^{1/2}\{{\bf I}-{\bf P}\}u$  into (\ref{balance}),
one can deduce the first system of macroscopic equations
\begin{equation} \label{macro1}
  \begin{cases}
    \partial_t a +\nabla_x\cdot b=0,\\[2mm]
    \partial_t b+\nabla_x(a+2c) +\nabla_x A(\{{\bf I}-{\bf P}\}u)+{\epsilon}b=0,\\[2mm]
    \partial_t c+\frac{1}{3}\nabla_x\cdot b+\frac{5}{3}\nabla_x\cdot B(\{{\bf I}-{\bf P}\}u)+2{\epsilon}c=0.
  \end{cases}
\end{equation}
To obtain the second system of macroscopic equations,
we split $u={\bf P}u+\{{\bf I}-{\bf P}\}u$ to decompose the equation (\ref{u}) as
\begin{equation} \label{Pu}
  \partial_t{\bf P}u+\xi\cdot\nabla_x{\bf P}u-{\epsilon}L_{FP}{\bf P}u
  =-\partial_t\{{\bf I}-{\bf P}\}u+R+G,
\end{equation}
with
\begin{equation}\label{R,G}
  R=-\xi\cdot\nabla_x \{{\bf{I}}-{\bf{P}}\}u+{\epsilon}L_{FP}\{{\bf{I}}-{\bf{P}}\}u+L\{{\bf{I}}-{\bf{P}}\}u,\quad G=\Gamma(u,u).
\end{equation}
Applying $A_{jm}(\cdot)$ and $B_j (\cdot)$ to both sides of (\ref{Pu}),
and using
\begin{equation*}
  \quad L_{FP}{\bf P}u=-b\cdot\xi M^{1/2}-2c(|\xi|^2-3)M^{1/2}
\end{equation*}
and the balance law of mass (\ref{macro1})$_1$, one has
\begin{equation} \label{macro2}
  \begin{cases}
    2\partial_jb_j+2\partial_t c+4{\epsilon}c=-\partial_tA_{jj}(\{{\bf{I}}-{\bf{P}}\}u)+A_{jj}(R+G),\\[2mm]
    \partial_jb_m+\partial_mb_j=-\partial_tA_{jm}(\{{\bf{I}}-{\bf{P}}\}u)+A_{jm}(R+G),\quad j\neq m,\\[2mm]
    \partial_tB_j(\{{\bf{I}}-{\bf{P}}\}u)+\partial_j c=B_j(R+G).
  \end{cases}
\end{equation}

Now we focus on the macroscopic equations (\ref{macro1}) and (\ref{macro2})
to estimate the higher order derivatives of the macroscopic coefficients $(a, b, c)$ in $L^2$ norm.
For this purpose, we first give a lemma without proofs.
Roughly speaking, the idea is just based on the fact
that the velocity-coordinate projector is bounded uniformly in $t$ and $x$,
and the velocity polynomials and velocity derivatives can be absorbed
by the global Maxwellian $M$ which exponentially decays in $\xi$.
\begin{lemma} \label{lemma1}
  For any $|\alpha|\leq N$ and $1\leq j,m\leq 3$, it holds that
  \begin{equation} \label{lemma1 1}
    \left\|\partial^{\alpha}_xA_{jm}(\{{\bf{I}}-{\bf{P}}\}u),\partial^{\alpha}_xB_j(\{{\bf{I}}-{\bf{P}}\}u)\right\|
    \lesssim \min\{\|\partial^{\alpha}_x\{{\bf{I}}-{\bf{P}}\}u\|,\|\partial^{\alpha}_x\{{\bf{I}}-{\bf{P}}\}u\|_{\nu}\}.
  \end{equation}
  Moreover, for any $|\alpha|\leq N-1$ and $1\leq j,m\leq 3$, it holds that
  \begin{equation} \label{lemma1 2}
    \|\partial^{\alpha}_xA_{jm}(R),\partial^{\alpha}_xB_j(R)\|
    \lesssim \sum_{|\alpha_1|\leq |\alpha|+1}\|\partial^{\alpha_1}_x\{{\bf{I}}-{\bf{P}}\}u\|_{\nu}
  \end{equation}
  and
  \begin{equation} \label{lemma1 3}
     \|\partial^{\alpha}_xA_{jm}(G),\partial^{\alpha}_xB_j(G)\|^2
     \lesssim \mathcal{E}_{q,l}(u)(t)\mathcal{D}_{q,l}(u)(t).
  \end{equation}
\end{lemma}

Next we state the key estimates on the macroscopic dissipation in the following theorem.
\begin{theorem}\label{macro_thm}
There is an interactive energy functional $\mathcal{E}_{int}(u)(t)$ such that
\begin{equation}\label{E_int}
  |\mathcal{E}_{int}(u)(t)|\lesssim \sum_{|\alpha|\leq N}\|\partial_x^{\alpha}u(t)\|^2
\end{equation}
and
\begin{equation}\label{macro}
  \begin{aligned}
    &\frac{d}{dt}\mathcal{E}_{int}(u)(t)
    +\lambda \sum_{|\alpha|\leq N-1}\|\partial^{\alpha}_x\nabla_{x}(a,b,c)(t)\|^2\\[2mm]
    \lesssim &\sum_{|\alpha|\leq N}\|\partial^{\alpha}_x\{{\bf {I}-\bf{P}}\}u\|_{\nu}^2
    +{\epsilon}^2\sum_{|\alpha|\leq N-1}\|\partial^{\alpha}_x(b,c)\|^2+\mathcal{E}_{q,l}(u)(t)\mathcal{D}_{q,l}(u)(t),
  \end{aligned}
\end{equation}
where $\mathcal{E}_{int}(u)(t)$ is the linear combination of the following terms over
$|\alpha|\leq N-1$ and $1\leq j\leq 3$:
\begin{align*}
  \mathcal{I}_{\alpha}^a(u(t))
  &=\langle\partial^{\alpha}_xb,\nabla_x\partial^{\alpha}_xa\rangle,\\[2mm]
  \mathcal{I}_{\alpha,j}^b(u(t))
  &=\left\langle\frac{1}{2}\sum_{m\neq j}\partial_j\partial^{\alpha}_xA_{mm}(\{{\bf{I}}-{\bf{P}}\}u)
     -\sum_m\partial_m\partial^{\alpha}_xA_{jm}(\{{\bf{I}}-{\bf{P}}\}u),\partial^{\alpha}_xb_j\right\rangle,\\[2mm]
  \mathcal{I}_{\alpha,j}^c(u(t))
  &=\left\langle\partial^{\alpha}_xB_j(\{{\bf{I}}-{\bf{P}}\}u),\partial_j\partial^{\alpha}_xc\right\rangle.
\end{align*}
\end{theorem}
\begin{proof}
  Step 1. \emph{Estimate on $b$.} For any $\eta>0$, it holds that
  \begin{equation} \label{b0}
    \begin{aligned}
    &\frac{d}{dt}\sum_{|\alpha|\leq N-1}\sum_{j}\mathcal{I}_{\alpha,j}^b(u(t))
    +\frac{1}{2}\sum_{|\alpha|\leq N-1}\|\partial_x^{\alpha}\nabla_xb\|^2\\[2mm]
    \leq & C\eta\sum_{|\alpha|\leq N-1}\|\partial^{\alpha}_x\nabla_x(a,c)\|^2
    +C\eta\sum_{|\alpha|\leq N-1}{\epsilon}^2\|\partial_x^{\alpha}b\|^2\\[2mm]
    &+C_{\eta}\sum_{|\alpha|\leq N}\|\partial^{\alpha}_x\{{\bf{I}}-{\bf{P}}\}u\|_{\nu}^2
    +C_{\eta}\mathcal{E}_{q,l}(u)(t)\mathcal{D}_{q,l}(u)(t).
  \end{aligned}
  \end{equation}
  In fact, for fixed $j\in\{1,2,3\}$, one can deduce from (\ref{macro2}) that
  \begin{equation}
       \begin{aligned}\label{elliptic}
         &-\Delta_xb_j-\partial_j\partial_jb_j\\[2mm]
        =&-\partial_t\left[\frac{1}{2}\sum_{m\neq j}\partial_jA_{mm}(\{{\bf I}-{\bf{P}}\}u)
         -\sum_m\partial_mA_{jm}(\{{\bf{I}}-{\bf{P}}\}u)\right]\\[2mm]
         &+\frac{1}{2}\sum_{m\neq j}\partial_jA_{mm}(R+G)-\sum_m\partial_mA_{jm}(R+G).
       \end{aligned}
  \end{equation}
  Let $|\alpha|\leq N-1$. Apply $\partial_x^{\alpha}$ to the elliptic-type equation (\ref{elliptic}),
  multiply it by $\partial_x^{\alpha}b_j$, and then integrate it over $\mathbb{R}^3$ to find
  \begin{equation}\label{b1}
  \begin{aligned}
    &\frac{d}{dt}\mathcal{I}_{\alpha,j}^b(u(t))
     +\|\nabla_x\partial^{\alpha}_xb_j\|^2+\|\partial_j\partial^{\alpha}_xb_j\|^2\\[2mm]
      =&\left\langle\frac{1}{2}\sum_{m\neq j}\partial_j\partial^{\alpha}_xA_{mm}(\{{\bf{I}}-{\bf{P}}\}u)
      -\sum_m\partial_m\partial^{\alpha}_xA_{jm}(\{{\bf{I}}-{\bf{P}}\}u),\partial^{\alpha}_x\partial_tb_j\right\rangle\\[2mm]
      &+\left\langle\frac{1}{2}\sum_{m\neq j}\partial_j\partial^{\alpha}_xA_{mm}(R+G)
      -\sum_m\partial_m\partial^{\alpha}_xA_{jm}(R+G),\partial^{\alpha}_xb_j\right\rangle\\[2mm]
      =&I_1^b+I_1^b.
  \end{aligned}
  \end{equation}
  Using (\ref{macro1})$_2$ (the second equation of (\ref{macro1}))
  to replace $\partial_tb_j$, we get
  \begin{equation}
    \label{I_1^b}
    \begin{aligned}
    I_1^b\leq &\eta \|\partial^{\alpha}_x\partial_tb_j\|^2
               +C_{\eta}\sum_{|\beta|\leq N}\left\|\partial_x^{\beta}A(\{{\bf I}-{\bf P}\}u)\right\|^2\\[2mm]
         \leq &4\eta \left\{\|\partial^{\alpha}_x\nabla_x(a,c)\|^2+{\epsilon}^2\|\partial_x^{\alpha}b_j\|^2\right\}
    +C_{\eta}\sum_{|\beta|\leq N}\left\|\partial^{\beta}_x\{{\bf{I}}-{\bf{P}}\}u\right\|_{\nu}^2.
    \end{aligned}
  \end{equation}
  Here we have used (\ref{lemma1}). For $I_2^b$, integrating by parts implies
  \begin{equation} \label{I_2^b}
    \begin{aligned}
       I_2^b=&-\frac{1}{2}\sum_{m\neq j}\left\langle\partial^{\alpha}_xA_{mm}(R+G),\partial_j\partial^{\alpha}_xb_j\right\rangle
              +\sum_m\left\langle\partial^{\alpha}_xA_{jm}(R+G),\partial_m\partial^{\alpha}_xb_j\right\rangle\\[2mm]
            \leq &  \frac{1}{2}\|\nabla_x\partial^{\alpha}_xb_j\|^2+C\sum_{m}\|\partial^{\alpha}_xA_{jm}(R,G)\|^2\\[2mm]
            \leq &  \frac{1}{2}\|\nabla_x\partial^{\alpha}_xb_j\|^2
                    +C\sum_{|\beta|\leq N} \left\|\partial^{\beta}_x\{{\bf{I}}-{\bf{P}}\}u\right\|_{\nu}^2
                    +C\mathcal{E}_{q,l}(u)(t)\mathcal{D}_{q,l}(u)(t).
    \end{aligned}
  \end{equation}
  Thus, (\ref{b0}) follows by plugging (\ref{I_1^b}) and (\ref{I_2^b}) into (\ref{b1}) and then taking summation
   over $1\leq j\leq 3$ and $|\alpha|\leq N-1$.

  Step 2. \emph{Estimate on $c$.} For any $\eta>0$, it holds that
  \begin{equation}\label{c0}
  \begin{aligned}
    &\frac{d}{dt}\sum_{|\alpha|\leq N-1}\sum_{j}\mathcal{I}_{\alpha,j}^c(u(t))
    +\frac{1}{2}\sum_{|\alpha|\leq N-1}\|\nabla_x\partial_x^{\alpha}c\|^2\\[2mm]
    \leq & 3\eta\sum_{|\alpha|\leq N-1}\|\partial^{\alpha}_x\nabla_xb\|^2
    +12\eta\sum_{|\alpha|\leq N-1}{\epsilon}^2\|\partial_x^{\alpha}c\|^2\\[2mm]
    &+C_{\eta}\sum_{|\alpha|\leq N}\|\partial^{\alpha}_x\{{\bf{I}}-{\bf{P}}\}u\|_{\nu}^2
    +C_{\eta}\mathcal{E}_{q,l}(u)(t)\mathcal{D}_{q,l}(u)(t).
  \end{aligned}
  \end{equation}
  Indeed, applying $\partial_x^{\alpha}$ with $|\alpha|\leq N-1$ to the macroscopic equation (\ref{macro2})$_3$,
  multiplying it by $\partial_j\partial_x^{\alpha}c$ and then integrating it over $\mathbb{R}^3$, we have
  \begin{equation}\label{c1}
  \begin{aligned}
    &\frac{d}{dt}\mathcal{I}_{\alpha,j}^c(u(t))+\|\partial_j\partial^{\alpha}_xc\|^2\\[2mm]
    =&\langle\partial^{\alpha}_xB_j(\{{\bf{I}}-{\bf{P}}\}u),\partial_t\partial_j\partial^{\alpha}_xc\rangle
    +\left\langle\partial^{\alpha}_xB_j(R+G),\partial_j\partial^{\alpha}_xc\right\rangle\\[2mm]
    =&I_1^c+I_2^c.
  \end{aligned}
  \end{equation}
  Use $(\ref{macro1})_3$ to replace $\partial_t c$ and estimate $I_1^c$ as
  \begin{equation} \label{I_1^c}
  \begin{aligned}
    I_1^c=&-\langle\partial_j\partial_{x}^{\alpha}B_j(\{{\bf I}-{\bf P}\}u),\partial_x^{\alpha}\partial_tc\rangle\\[2mm]
    \leq &\eta\|\partial_x^{\alpha}\partial_t c\|^2+C_{\eta}\|\partial_j\partial_{x}^{\alpha}B_j(\{{\bf I}-{\bf P}\}u)\|^2\\[2mm]
    \leq &\eta\left\{\|\partial^{\alpha}_x\nabla_xb\|^2+4\epsilon^2(t)\|\partial_x^{\alpha}c\|^2\right\}
    +C_{\eta}\sum_{|\beta|\leq N}\left\|\partial^{\beta}_x\{{\bf{I}}-{\bf{P}}\}u\right\|_{\nu}^2.
  \end{aligned}
  \end{equation}
  $I_2^c$ is bounded by
  \begin{equation} \label{I_2^c}
  \begin{aligned}
      I_2^c\leq &\frac{1}{2}\|\partial_j\partial^{\alpha}_xc\|^2
  +\|\partial^{\alpha}_xB_j(R,G)\|^2\\[2mm]
   \leq &\frac{1}{2}\|\partial_j\partial^{\alpha}_xc\|^2+C\mathcal{E}_{q,l}(u)(t)\mathcal{D}_{q,l}(u)(t).
  \end{aligned}
  \end{equation}
  Thus, (\ref{c0}) follows
  by plugging (\ref{I_1^c}) and (\ref{I_2^c}) into $(\ref{c1})$,
  and summing it over $1\leq j\leq 3$ and $|\alpha|\leq N-1$.

  Step 3. \emph{Estimate on $a$.}
  Let $|\alpha|\leq N-1$. Apply $\partial_x^{\alpha}$ to(\ref{macro1})$_2$,
  multiply it by $\partial_x^{\alpha}\nabla_x a$ and then integrate it over $\mathbb{R}^3$ to discover
  \begin{equation}\label{a1}
  \begin{aligned}
    &\partial_t\langle\partial^{\alpha}_xb,\partial^{\alpha}_x\nabla_x a\rangle
    +\|\partial^{\alpha}_x\nabla_xa\|^2\\[2mm]
    =&-\langle2\partial^{\alpha}_x\nabla_xc+\partial^{\alpha}_x\nabla_xA(\{{\bf I}-{\bf P}\}u)
    +{\epsilon}\partial^{\alpha}_xb,\partial^{\alpha}_x\nabla_xa\rangle
    +\langle\partial^{\alpha}_xb,\partial_t\partial^{\alpha}_x\nabla_xa\rangle\\[2mm]
    =&-\langle2\partial^{\alpha}_x\nabla_xc+\partial^{\alpha}_x\nabla_xA(\{{\bf I}-{\bf P}\}u)
    +{\epsilon}\partial^{\alpha}_xb,\partial^{\alpha}_x\nabla_xa\rangle
    +\langle\partial^{\alpha}_x\nabla_x\cdot b,\partial^{\alpha}_x\nabla_x\cdot b\rangle\\[2mm]
    \leq& \frac{1}{2}\|\partial^{\alpha}_x\nabla_xa\|^2
    +{\epsilon}^2\|\partial^{\alpha}_xb\|^2
    +C\|\partial^{\alpha}_x\nabla_x(b,c)\|^2
    +C\|\partial^{\alpha}_x\nabla_xA(\{{\bf I}-{\bf P}\}u)\|^2.
  \end{aligned}
  \end{equation}
Here we used the conservation of mass (\ref{macro1})$_1$.
Take summation (\ref{a1}) over $|\alpha|\leq N-1$ to get
\begin{equation}\label{a0}
  \begin{aligned}
    &\frac{d}{dt}\sum_{|\alpha|\leq N-1}\langle\partial^{\alpha}_xb,\nabla_x\partial^{\alpha}_xa\rangle
    +\frac{1}{2}\sum_{|\alpha|\leq N-1}\|\partial^{\alpha}_x\nabla_xa\|^2\\[2mm]
    \lesssim&\sum_{|\alpha|\leq N-1}\|\partial^{\alpha}_x\nabla_x(b,c)\|^2+
        \sum_{|\alpha|\leq N-1}{\epsilon}^2\|\partial^{\alpha}_xb\|^2
    +C\sum_{|\alpha|\leq N}\|\partial^{\alpha}_x\{{\bf {I}-\bf{P}}\}u\|_{\nu}^2.
  \end{aligned}
\end{equation}

Step 4. \emph{Combination.}
We have finished the estimates of $a,b,c$.
With them in hand, let us multiply $(\ref{b0})$ and $(\ref{c0})$ by a constant $M>0$
and take summation of both of them as well as (\ref{a0}).
One can first choose $M>0$ sufficiently large such that the first term on the right-hand side
of (\ref{a0}) can be absorbed
by the dissipation of $b$ and $c$. By fixing $M>0$, one can choose $\eta>0$ sufficiently small
such that the first terms on the right-hand side of  $(\ref{b0})$ and $(\ref{c0})$
are absorbed by the full dissipation of $b$ and $c$. Hence, we have proved (\ref{macro}).
Cauchy's inequality and (\ref{lemma1 1}) yield
\begin{equation*}
  |\mathcal{E}_{int}(u)(t)|\lesssim \sum_{|\alpha|\leq N-1}\left\{\|\partial^{\alpha}_x\nabla_{x}(a,b,c)(t)\|^2
  +\|\partial^{\alpha}_x\{{\bf {I}-\bf{P}}\}u\|^2+\|\partial_{x}^{\alpha} b\|^2\right\},
\end{equation*}
which implies (\ref{E_int}).
Therefore one has finished the proof of Theorem 2.1.
\end{proof}

\section{Global Existence}
In this section,  we shall devote ourselves to obtaining the existence of classical solutions to (\ref{u})
globally in time. For this purpose, we first collect some estimates for the linearized Fokker-Planck operator $L_{FP}$
and the collision operators $L$ and $\Gamma$.

For the linearized Fokker-Planck operator $L_{FP}$, we have the following two results. The first one is concerned with the dissipative property of the
linearized Fokker-Planck operator $L_{FP}$ without weight
\begin{lemma} (\cite{Arnold01}, \cite{Duan10})
  $L_{FP}$ is a linear self-adjoint operator with respect to the duality induced by the $L_{\xi}^2$-scalar product.
  Furthermore, there exists a constant $\lambda_{FP} > 0$ such that
  \begin{equation} \label{L_FP0}
    -( u, L_{FP}u) \geq \lambda_{FP}\|\{{\bf I}-{\bf P}_0\}u\|^2.
  \end{equation}
\end{lemma}
For the dissipative property of the
linearized Fokker-Planck operator $L_{FP}$ with the weight $w_q^{l-|\beta|}$, we have
\begin{lemma} It holds that for any $l\geq 0$,
  \begin{equation} \label{Lfp w1}
    \begin{aligned}
      \left(L_{FP}\partial_{\beta}^{\alpha}u,w_q^{2(l-|\beta|)}\partial_{\beta}^{\alpha}u\right)
      \leq- \frac{1}{2}\lambda_{FP}
      \left\|\{{\bf{I}}-{\bf{P}}_0\}(w_{q}^{l-|\beta|}\partial_{\beta}^{\alpha}u)\right\|^2
      +C(q-\gamma)^{2}\left\|w_{q}^{l-|\beta|}\partial_{\beta}^{\alpha}u\right\|_{\nu}^2.
    \end{aligned}
  \end{equation}
\end{lemma}
\begin{proof} Integrating by parts yields
  \begin{equation}\label{Lfp e}
  \begin{aligned}
    & \left(L_{FP}\partial_{\beta}^{\alpha}u,w^{2(l-|\beta|)}\partial_{\beta}^{\alpha}u\right)
    -\left( L_{FP}(w_{q}^{l-|\beta|}\partial_{\beta}^{\alpha}u),
    w_{q}^{l-|\beta|}\partial_{\beta}^{\alpha}u\right)\\[2mm]
    =&-\left(\nabla_{\xi}\cdot\left(\partial_{\beta}^{\alpha}u\nabla_{\xi}w_{q}^{l-|\beta|}\right)
    +\nabla_{\xi}w_{q}^{l-|\beta|}\cdot\nabla_{\xi}\partial_{\beta}^{\alpha}u,
    w_{q}^{l-|\beta|}\partial_{\beta}^{\alpha}u\right)\\[2mm]
    =&\left(\nabla_{\xi}w_{q}^{l-|\beta|}\partial_{\beta}^{\alpha}u,\nabla_{\xi}(w_{q}^{l-|\beta|}\partial_{\beta}^{\alpha}u)\right)
    -\left(\nabla_{\xi}w_{q}^{l-|\beta|}\cdot\nabla_{\xi}\partial_{\beta}^{\alpha}u,
    w_{q}^{l-|\beta|}\partial_{\beta}^{\alpha}u\right)\\[2mm]
    =&\left( \nabla_{\xi}w_{q}^{l-|\beta|}\partial_{\beta}^{\alpha}u,
    \nabla_{\xi}w_{q}^{l-|\beta|}\partial_{\beta}^{\alpha}u\right)\\[2mm]
    \leq & C(q-\gamma)^2\left\|\left(\chi_{|\xi|>R}+\chi_{|\xi| \leq R}\right)\langle \xi \rangle^{-1}w_{q}^{l-|\beta|}\partial_{\beta}^{\alpha}u\right\|^2
  \end{aligned}
  \end{equation}
  for each $R>0$.
  Here, we have used the fact that
  $$
  \nabla_{\xi}w_{q}^{l-|\beta|}
  =(q-|\beta|)(1-\gamma)w_{q}^{l-|\beta|}\frac{\xi}{1+|\xi|^2}.
  $$
  We estimate the terms on the right hand side of (\ref{Lfp e}). First,
  \begin{equation} \label{>R}
    \begin{aligned}
      &\left\|\chi_{|\xi|>R}\langle \xi \rangle^{-1}w_{q}^{l-|\beta|}\partial_{\beta}^{\alpha}u\right\|^2\\[2mm]
       \leq& R^{-2}\left\|\{{\bf  I}-{\bf  P}_0\}
       \left(w_{q}^{l-|\beta|}\partial_{\beta}^{\alpha}u\right)\right\|^2
      +C\left\|{\bf  P}_0\left(w_{q}^{l-|\beta|}\partial_{\beta}^{\alpha}u\right)\right\|^2\\[2mm]
      \leq & R^{-2}\left\|\{{\bf  I}-{\bf  P}_0\}
         \left(w_{q}^{l-|\beta|}\partial_{\beta}^{\alpha}u\right)\right\|^2 +C\left\|w_{q}^{l-|\beta|}\partial_{\beta}^{\alpha}u\right\|_{\nu}^2.
    \end{aligned}
  \end{equation}
  If $\xi$ is bounded, then $\langle \xi \rangle^{-2}\sim \nu(\xi)$ which implies
  \begin{equation}
    \label{<R}
    \left\|\chi_{|\xi|\leq R}\langle \xi \rangle^{-1}w_{q}^{l-|\beta|}\partial_{\beta}^{\alpha}u\right\|^2
    \lesssim  \left\|w_{q}^{l-|\beta|}\partial_{\beta}^{\alpha}u\right\|_{\nu}^2.
  \end{equation}
  Plugging (\ref{>R}) and (\ref{<R}) into (\ref{Lfp e}), and noticing that
  \begin{equation*}
     -\left( L_{FP}(w_{q}^{l-|\beta|}\partial_{\beta}^{\alpha}u),
    w_{q}^{l-|\beta|}\partial_{\beta}^{\alpha}u\right)
    \geq \lambda_{FP}
    \left\|\{{\bf{I}}-{\bf{P}}_0\}(w_{q}^{l-|\beta|}\partial_{\beta}^{\alpha}u)\right\|^2
  \end{equation*}
  from (\ref{L_FP0}),
  one can prove (\ref{Lfp w1}) by choosing $R>0$ sufficiently large.
\end{proof}

For the corresponding weighed estimates on the linearized Boltzmann collision operator $L$ and the nonlinear collision operator $\Gamma$, we have
\begin{lemma} \label{collision} (\cite{Guo06}, \cite{Guo03})
  Consider the inverse power law with $-3<\gamma\leq 1$.
  If $\eta>0$ and $m\geq 0$,
  then there are $C_{\eta}, C>0$, such that
  \begin{align}
  \label{L1}
   -\left(\langle \xi \rangle^{2m}\partial_{\beta}Lg,\partial_{\beta}g \right)
  \geq &~~\frac{1}{2}\left\|\langle \xi \rangle^{m}\partial_{\beta}g\right\|_{\nu}^2
          -\eta\sum_{|\beta_1|\leq|\beta| }\left\|\langle \xi \rangle^{m}\partial_{\beta_1}g\right\|_{\nu}^2
          -C_{\eta}\|g\|_{\nu}^2,\\[2mm]
  \label{Gamma0}
  \left|\left\langle \langle \xi \rangle^{2m}\partial_{\beta}\Gamma(f_1,f_2),\partial_{\beta}h \right\rangle\right|
  \lesssim &\sum_{i,j}\sum_{\beta_1+\beta_2\leq \beta}
          \left|\langle \xi \rangle^{m}\partial_{\beta_1}f_i\right|\left|\langle \xi \rangle^{m}\partial_{\beta_2}f_j\right|_{\nu}
               \left|\langle \xi \rangle^{m}\partial_{\beta}h\right|_{\nu}.
  \end{align}
\end{lemma}

\begin{lemma} It holds that for any $l\geq 0$,
  \begin{align}
    \label{Gamma w1}
    \left(\partial_x^{\alpha}\Gamma(u,u),w_q^{2l}\partial_x^{\alpha}u\right)
    \lesssim &~\mathcal{E}_{q,l}(u)^{1/2}(t)\mathcal{D}_{q,l}(u)(t),\\[2mm]
    \label{Gamma w0}
    (\partial_{\beta}^{\alpha}\Gamma(u,u),w_q^{2(l-|\beta|)}\partial_{\beta}^{\alpha}\{{\bf{I}}-{\bf{P}}\}u)
    \lesssim  &~\mathcal{E}_{q,l}(u)^{1/2}(t)\mathcal{D}_{q,l}(u)(t).
  \end{align}
\end{lemma}


Next, as the first step,
we shall obtain the dissipation rate
$$
{\epsilon}\sum_{|\alpha|\leq N}\|\{{\bf I}-{\bf P}_0\}\partial_x^{\alpha} u\|^2.
$$
To this end, we consider the non-weighted energy estimates on the solution $u$ of (\ref{u})-(\ref{u_0}).
Taking $\partial^{\alpha}_x$ of the equation (\ref{u}) yields
\begin{equation} \label{nonweight}
  \frac{1}{2}\frac{d}{dt}\|\partial_x^{\alpha}u\|^2-(L\partial_x^{\alpha}u,\partial_x^{\alpha}u)
  -{\epsilon}(L_{FP}\partial_x^{\alpha}u,\partial_x^{\alpha}u)
  =(\partial_x^{\alpha}\Gamma(u,u),\partial_x^{\alpha}u).
\end{equation}
Applying (\ref{L0}), (\ref{L_FP0}) and (\ref{Gamma w1}) with $l=0$ to (\ref{nonweight}),
we thus get the following lemma.
\begin{lemma}\label{pure-x}
  It holds that for each $t>0$,
  \begin{equation}\label{nonweight0}
    \begin{aligned}
      \frac{1}{2}\frac{d}{dt}\sum_{|\alpha|\leq N}\|\partial^{\alpha}_xu\|^2
      +\lambda_0\sum_{|\alpha|\leq N}\|\partial^{\alpha}_x\{{\bf {I}-\bf{P}}\}u\|^2_{\nu}&\\[2mm]
      +\lambda_{FP}{\epsilon}\sum_{|\alpha|\leq N}\|\{{\bf I}-{\bf P}_0\}\partial_x^{\alpha} u\|^2
      & \lesssim \mathcal{E}_{q,l}(u)^{1/2}(t)\mathcal{D}_{q,l}(u)(t).
    \end{aligned}
  \end{equation}
\end{lemma}

For the second step, we consider the weighted energy estimates on $u$ to get the dissipation rate
$$\sum_{|\alpha|+|\beta|\leq N}\|w_{q}^{l-|\beta|}
                           \partial_{\beta}^{\alpha}\{{\bf I-P}\}u(t)\|_{\nu}^2.$$
\begin{lemma}
 There is a positive constant $\delta_0$ such that if
 \begin{equation}
   \label{assumption}
   \sup_{0\leq t\leq T}\mathcal{E}_{q,l}(u)(t)\leq\delta_0
 \end{equation}
 and $(q-\gamma)^2{\epsilon}\leq \delta_0$, then
 \begin{equation}
   \label{Lyapunov}
   \frac{d}{dt}\mathcal{E}_{q,l}(u)(t)+\lambda \mathcal{D}_{q,l}(u)(t)\leq 0.
 \end{equation}
\end{lemma}
\begin{proof}
  \textbf{Step 1.} Weight estimate on zero-order of $\{{\bf {I}}-{\bf{P}}\}u$:
  \begin{equation}\label{0-order}
  \begin{aligned}
    &\frac{d}{dt}\|w_q^l(\xi)\{{\bf{I}}-{\bf{P}}\}u(t)\|^2
    +\frac{1}{2}\|w_{q}^{l}\{{\bf{I}}-{\bf{P}}\}u\|_{\nu}^2\\[2mm]
    &+{\epsilon}\lambda_{FP}\|\{{\bf{I}}-{\bf{P}}_0\}(w_q^l\{{\bf{I}}-{\bf{P}}\}u)(t)\|^2\\[2mm]
    \lesssim &\|\{{\bf{I}}-{\bf{P}}\}u\|_{\nu}^2+\|\nabla_xu\|_{\nu}^2
    +\mathcal{E}_{q,l}(u)^{1/2}(t)\mathcal{D}_{q,l}(u)(t).
  \end{aligned}
  \end{equation}
  In fact, apply $\{{\bf {I}}-{\bf{P}}\}$ to (\ref{u}) and then use
  $$ L_{FP}{\bf{P}}u={\bf{P}}L_{FP}u$$
  to find
  \begin{equation}\label{(I-P)u}
  \begin{aligned}
    &\partial_t\{{\bf {I}}-{\bf{P}}\}u+\xi\cdot\nabla_x\{{\bf {I}}-{\bf{P}}\}u-L\{{\bf {I}}-{\bf{P}}\}u\\[2mm]
    =&\Gamma(u,u)+{\epsilon}L_{FP}\{{\bf{I}}-{\bf{P}}\}u+{\bf{P}}\xi\cdot\nabla_xu-
    \xi\cdot\nabla_x{\bf{P}}u.
    \end{aligned}
  \end{equation}
  Multiply (\ref{(I-P)u}) by $w_{q}^{2l}\{{\bf{I}}-{\bf{P}}\}u$ and integrate it over $\mathbb{R}^3\times\mathbb{R}^3$
  to have
  \begin{equation} \label{0-order1}
  \begin{aligned}
    &\frac{1}{2}\frac{d}{dt}\|w_{q}^{l} \{{\bf{I}}-{\bf{P}}\}u\|^2-(w_{q}^{2l}L\{{\bf{I}}-{\bf{P}}\}u,\{{\bf{I}}-{\bf{P}}\}u)\\[2mm]
    =&(w_{q}^{2l}\Gamma(u,u),\{{\bf{I}}-{\bf{P}}\}u)+{\epsilon}(L_{FP}\{{\bf{I}}-{\bf{P}}\}u,w_{q}^{2l}\{{\bf{I}}-{\bf{P}}\}u)\\[2mm]
    &+({\bf P}\xi\cdot\nabla_xu-\xi\cdot\nabla_x{\bf P}u,w_{q}^{2l}\{{\bf{I}}-{\bf{P}}\}u).
  \end{aligned}
  \end{equation}
  Cauchy's inequality yields that
  the third term on the right-hand side of (\ref{0-order1}) is bounded by
  $$
  \frac{1}{8}\|w_{q}^{l}\{{\bf{I}}-{\bf{P}}\}u\|_{\nu}^2
  +C\|\nabla_xu\|_{\nu}^2.
  $$
  Plugging (\ref{L1}), (\ref{Gamma w0}) and (\ref{Lfp w1}) into (\ref{0-order1}),
  we can prove  (\ref{0-order}) when $(q-\gamma)^2{\epsilon}$ is suitably small.
  \\
  {\bf Step 2.} Weighted estimate on pure space-derivative of $u$:
  \begin{equation}
    \label{x weight}
    \begin{aligned}
      &\frac{d}{dt}\sum_{1\leq |\alpha|\leq N}\|w_{q}^{l}\partial_x^{\alpha}u\|^2
      +\frac{1}{2}\sum_{1\leq |\alpha|\leq N}\|w_{q}^{l}\partial_x^{\alpha}u\|_{\nu}^2\\[2mm]
      &+\lambda_{FP}{\epsilon}\sum_{1\leq |\alpha|\leq N}
      \|\{{\bf I}-{\bf P}_0\}(w_{q}^{l}\partial_x^{\alpha} u)\|^2\\[2mm]
      \lesssim &
      \sum_{1\leq |\alpha|\leq N}\|\partial_x^{\alpha}u\|_{\nu}^2+\mathcal{E}_{q,l}(u)^{1/2}(t)\mathcal{D}_{q,l}(u)(t).
    \end{aligned}
  \end{equation}
  In fact, let $1\leq|\alpha|\leq N$. Taking $\partial_x^{\alpha}$ of (\ref{u}), multiplying it
  by $w_{q}^{2l}(\xi)\partial_x^{\alpha}u$, and then integrating it over $\mathbb{R}^3\times\mathbb{R}^3$,
  one has
  \begin{equation} \label{x weight1}
    \begin{aligned}
      &\frac{1}{2}\frac{d}{dt}\|w_{q}^{l}\partial_{x}^{\alpha}u\|^2
      -(w_{q}^{2l}L\partial_{x}^{\alpha}u,\partial_{x}^{\alpha}u)\\[2mm]
     =&(\partial_{x}^{\alpha}\Gamma(u,u),w_{q}^{2l}\partial_{x}^{\alpha}u)
      +{\epsilon}(L_{FP}\partial_{x}^{\alpha}u,w_{q}^{2l}\partial_{x}^{\alpha}u).
    \end{aligned}
  \end{equation}
  Hence, (\ref{x weight}) follows from plugging the estimates
  (\ref{L1}), (\ref{Gamma w1}) and (\ref{Lfp w1}) into (\ref{x weight1}) and then taking
  summation over $1\leq |\alpha|\leq N$.

  {\bf Step 3.} Weighted estimate on mixed space-velocity-derivative of $u$:
  \begin{equation}
    \label{mixed}
    \begin{aligned}
      &\frac{d}{dt}\sum_{m=1}^{N}C_m\sum_{\substack{|\beta|=m\\|\alpha|+|\beta|\leq N}}
      \|w_{q}^{l-|\beta|}\partial_{\beta}^{\alpha}\{{\bf{I}}-{\bf{P}}\}u\|^2\\[2mm]
      &+\lambda\sum_{\substack{|\beta|\geq 1\\ |\alpha|+|\beta|\leq N}}
      \left\{\|w_{q}^{l-|\beta|}\partial_{\beta}^{\alpha}\{{\bf{I}}-{\bf{P}}\}u\|_{\nu}^2
      +{\epsilon}\|\{{\bf I}-{\bf P}_0\}(w_{q}^{l-|\beta|}\partial_{\beta}^{\alpha}
      \{{\bf{I}}-{\bf{P}}\}u)\|^2\right\}\\[2mm]
      \lesssim &
      \sum_{1\leq |\alpha|\leq N}\|\partial_x^{\alpha}u\|_{\nu}^2+\mathcal{E}_{q,l}(u)^{1/2}(t)\mathcal{D}_{q,l}(u)(t).
    \end{aligned}
  \end{equation}
  Indeed, let $|\beta|=m>0$ and $|\alpha|+|\beta|\leq N$.
  For notational simplicity, we denote that $u_2\equiv \{{\bf{I}}-{\bf{P}}\}u$.
  Apply $\partial_{\beta}^{\alpha}$ to (\ref{(I-P)u}),
  and multiply it by $w_q^{2(l-|\beta|)}\partial_{\beta}^{\alpha}u_2$
  and then integrate over $\mathbb{R}^3\times\mathbb{R}^3$ to find
  \begin{equation}
    \label{mixed1}
    \begin{aligned}
      &\frac{1}{2}\frac{d}{dt}\left\|w_{q}^{l-|\beta|}\partial_{\beta}^{\alpha}u_2\right\|^2
      -\left(w_q^{2(l-|\beta|)}\partial_{\beta}^{\alpha}Lu_2,\partial_{\beta}^{\alpha}u_2\right)\\[2mm]
     =&\left(\partial_{\beta}^{\alpha}\Gamma(u,u),w_q^{2(l-|\beta|)}\partial_{\beta}^{\alpha}u_2\right)
      +{\epsilon}\left(\partial_{\beta}^{\alpha}L_{FP}u_2,
      w_q^{2(l-|\beta|)}\partial_{\beta}^{\alpha}u_2\right)\\[2mm]
      &-\left(\partial_{\beta}^{\alpha}(\xi\cdot\nabla_xu_2),
      w_q^{2(l-|\beta|)}\partial_{\beta}^{\alpha}u_2\right)\\[2mm]
      &+\left(\partial_{\beta}^{\alpha}({\bf P}\xi\cdot\nabla_x u-\xi\cdot\nabla_x{\bf P}u),
      w_q^{2(l-|\beta|)}\partial_{\beta}^{\alpha}u_2\right).
    \end{aligned}
  \end{equation}
  Noting that $w_{q}^{l-|\beta|}\leq w_q^{l-|\beta_1|}$ whenever $|\beta_1|\leq |\beta|$,
  we obtain from (\ref{L1}) that
  \begin{equation}
    \label{mixed2}
    \begin{aligned}
      -\left(w_q^{2(l-|\beta|)}\partial_{\beta}^{\alpha}Lu_2,\partial_{\beta}^{\alpha}u_2 \right)
   \geq ~\frac{1}{2}\left\|w_{q}^{l-|\beta|}\partial_{\beta}^{\alpha}u_2\right\|_{\nu}^2
          -C_{\eta}\|\partial_{x}^{\alpha}u_2\|_{\nu}^2
          -\eta\sum_{|\beta_1|\leq |\beta|}
          \left\|w_q^{l-|\beta_1|}\partial_{\beta_1}^{\alpha}u_2\right\|_{\nu}^2.
    \end{aligned}
  \end{equation}
  We estimate the terms on the right hand side of (\ref{mixed1}).
  Recall that $w=\langle \xi \rangle^{q-\gamma},$ which implies that
  $$\langle \xi \rangle\lesssim \nu(\xi)w^{-1}(\xi),$$
  whenever $q\geq1$.
  Hence we have
  \begin{equation} \label{mixed3}
  \begin{aligned}
        &(\partial_{\beta}^{\alpha}L_{FP}u_2,w_q^{2(l-|\beta|)}\partial_{\beta}^{\alpha}u_2)
         -(L_{FP}\partial^{\alpha}_{\beta}u_2,w_q^{2(l-|\beta|)}\partial_{\beta}^{\alpha}u_2)\\[2mm]
   =    &-\frac{1}{4}\sum_{0<\beta_1\leq \beta}C_{\beta}^{\beta_1}
         (\partial_{\beta_1}|\xi|^2\partial^{\alpha}_{\beta-\beta_1}u_2,
         w_q^{2(l-|\beta|)}\partial_{\beta}^{\alpha}u_2)\\[2mm]
   \leq &C \sum_{0<\beta_1\leq \beta}
         (\langle \xi \rangle|\partial^{\alpha}_{\beta-\beta_1}u_2|,w_q^{2(l-|\beta|)}|\partial_{\beta}^{\alpha}u_2|)\\[2mm]
   \leq &C \sum_{0<\beta_1\leq \beta}
         (\nu(\xi)w^{l-|\beta-\beta_1|}|\partial^{\alpha}_{\beta-\beta_1}u_2|,w_{q}^{l-|\beta|}
         |\partial_{\beta}^{\alpha}u_2|)\\[2mm]
   \leq &\eta \left\|w_{q}^{l-|\beta|}\partial^{\alpha}_{\beta}u_2\right\|_{\nu}^2
         +C_{\eta}\sum_{|\beta_1|<m}
          \left\|w_q^{l-|\beta_1|}\partial^{\alpha}_{\beta_1}u_2\right\|_{\nu}^2.
  \end{aligned}
  \end{equation}
  For the third term on the right hand side of (\ref{mixed1}),
  \begin{equation} \label{mixed4}
    \begin{aligned}
          &(\partial_{\beta}^{\alpha}(\xi\cdot\nabla_xu_2),w_q^{2(l-|\beta|)}\partial_{\beta}^{\alpha}u_2)\\[2mm]
      =   &(\partial_{\beta}^{\alpha}(\xi\cdot\nabla_xu_2),w_q^{2(l-|\beta|)}\partial_{\beta}^{\alpha}u_2)
           -(\xi\cdot\nabla_x\partial_{\beta}^{\alpha}u_2,w_q^{2(l-|\beta|)}\partial_{\beta}^{\alpha}u_2)\\[2mm]
      =   &\sum_{|\beta_1|=1}C_{\beta}^{\beta_1}
           (\partial_{\beta-\beta_1}^{\alpha+\beta_1}u_2,w_q^{2(l-|\beta|)}\partial_{\beta}^{\alpha}u_2)\\[2mm]
      \leq&\eta \left\|w_{q}^{l-|\beta|}\partial^{\alpha}_{\beta}u_2\right\|_{\nu}^2
         +C_{\eta}\sum_{\substack{|\beta_1|<m\\|\alpha_1|+|\beta_1|\leq N}}
          \left\|w_q^{l-|\beta_1|}\partial^{\alpha_1}_{\beta_1}u_2\right\|_{\nu}^2.
    \end{aligned}
  \end{equation}
  The last term on the right hand side of (\ref{mixed}) is bounded by
  \begin{equation} \label{mixed5}
    \begin{aligned}
      &\left(\partial_{\beta}^{\alpha}({\bf P}\xi\cdot\nabla_x u-\xi\cdot\nabla_x{\bf P}u),
      w_q^{2(l-|\beta|)}\partial_{\beta}^{\alpha}u_2\right)\\[2mm]
      \leq & \eta\left\|w_q^{l-|\beta|}\partial_{\beta}^{\alpha}u_2\right\|^2_{\nu}
      +C_{\eta}\sum_{1\leq |\alpha|\leq N}\|\partial_x^{\alpha}u\|_{\nu}^2.
    \end{aligned}
  \end{equation}
  Therefore, by choosing a small constant $\eta>0$, (\ref{mixed}) follows by
  plugging the estimates (\ref{mixed2}), (\ref{Gamma w1}), (\ref{mixed3}), (\ref{Lfp w1}), (\ref{mixed4})
  and (\ref{mixed5})
  into $(\ref{mixed1})$, taking summation over $\{|\beta|=m,|\alpha|+|\beta|\leq N\}$
  for each given $1\leq m \leq N$ and taking proper linear combination
  of those $N-1$ estimates with properly chosen constants $C_m>0 (1\leq m\leq N)$.

  \textbf{Step 4.} \emph{Combination.}
  First, let us multiply (\ref{nonweight0}) by a constant $M_1>0$ and sum it with (\ref{macro}).
  Note that  it holds that (\ref{E_int}) and
  $$\sum_{|\alpha|\leq N}\|\partial^{\alpha}_x(b,c)\|^2\leq\sum_{|\alpha|\leq N} \|\{{\bf I}-{\bf P}_0\}\partial_x^{\alpha} u\|^2.$$
   Thus, one can take $M_1 > 0$ such that the terms on the right-hand side of (\ref{macro}) can be absorbed
  and
  $$\mathcal{E}_{int}(u)(t)+\frac{1}{2}M_1\sum_{|\alpha|\leq N}\|\partial^{\alpha}_xu\|^2
  \sim \sum_{|\alpha|\leq N}\|\partial^{\alpha}_xu\|^2.$$
  In the further linear combination
  $$
  (\ref{0-order})+(\ref{x weight})+(\ref{mixed})+M_2\times[M_1\times(\ref{nonweight0})+(\ref{macro})],
  $$
  one can take $M_2 > 0$ large enough to absorb all the dissipation terms on the right-hand sides of
  (\ref{0-order}), (\ref{x weight}) and (\ref{mixed}), which implies
  \begin{equation}
   \label{weightE}
   \frac{d}{dt}\mathcal{E}_{q,l}(u)(t)+\lambda \mathcal{D}_{q,l}(u)(t)
   \lesssim \left[\mathcal{E}_{q,l}(u)^{1/2}(t)+\mathcal{E}_{q,l}(u)(t)\right]\mathcal{D}_{q,l}(u)(t).
   \end{equation}
  Therefore, (\ref{Lyapunov}) follows under the a priori assumption (\ref{assumption}).
  \end{proof}

  {\bf Proof of Theorem \ref{theorem1}(i) and Theorem \ref{theorem2}(i):}
  Fix $N$, $l$ as stated in Theorem 1.1 or Theorem 1.2.
  The local existence and uniqueness of the solution $u(t,x,\xi)$
  to the Cauchy problem $(\ref{u})-(\ref{u_0})$ can be proved in terms of the energy functional $\mathcal{E}_{q,l}(u)(t)$
  given by $(\ref{E})$, and the details are omitted for simplicity, see $\cite{Guo02, Guo03,Li_Matsumura}$ with a little modification.
  Now we have obtained the unform-in-time estimate $(\ref{Lyapunov})$ over $0\leq t\leq T$ with $0<T\leq \infty$.
  By the standard continuity argument, the global existence follows provided the initial energy functional
  $\mathcal{E}(u_0)$ is sufficiently small.

\section{Time Decay}
\subsection{The hard potential case}

In this subsection, we devote ourselves to obtaining the time decay rate of the global solution $u$ to the
Fokker-Planck-Boltzmann equation $(\ref{u})$-$(\ref{u_0})$ in the hard potential case ($0\leq \gamma\leq1$).
For this purpose, we first deduce some estimates for the Cauchy problem:
\begin{equation}\label{linear-equation}
  \begin{cases}
  \partial_tu+\xi\cdot\nabla_xu=Lu+{\epsilon}L_{FP}u+G,\\[2mm]
  u(0,x,\xi)=u_0(x,\xi),
\end{cases}
\end{equation}
where $u_0(x,\xi)$ and $G=G(t,x,\xi)$ with ${\bf P}G=0$ are given.
Formally, the solution $u$ to the Cauchy problem
$(\ref{linear-equation})$ can be written as the mild form
\begin{equation*}
  u(t)={e^{tB}}u_0+\int_0^te^{(t-s)B}h(s)ds,
\end{equation*}
where ${e^{tB}}$ denotes the solution operator to the Cauchy problem
of (\ref{linear-equation}) with $G\equiv 0$.
We first show that the operator ${e^{tB}}$
has the proposed algebraic decay properties as time tends to infinity.
The idea of the proofs is to make energy estimates for pointwise time $t$ and
frequency variable $k$, which corresponds to the spatial variable $x$.

\begin{lemma}\label{macro-thm}
There is $M>0$ such that
the free energy functional $\mathcal{E}_{free}(\widehat{u})(t,k)$, defined by
\begin{equation}\label{E_free}
  \begin{aligned}
    \mathcal{E}_{free}(\widehat{u})(t,k)=
    &M\sum_{j}\Bigg(\frac{1}{2}\sum_{m\neq j}\frac{ik_j}{1+|k|^2}A_{mm}(\{{\bf{I}}-{\bf{P}}\}\widehat{u})
    -\sum_{m}\frac{ik_m}{1+|k|^2}A_{jm}(\{{\bf{I}}-{\bf{P}}\}\widehat{u})|-\widehat{b}_j\Bigg)\\[2mm]
    &+M\sum_{j}\left(B_j\{{\bf{I}}-{\bf{P}}\}\widehat{u}|\frac{ik_j\widehat{a}}{1+|k|^2}\right)
    +\sum_{j}\left(\widehat{b}_j|\frac{ik_j}{1+|k|^2}\widehat{a}\right)
  \end{aligned}
\end{equation}
satisfies
\begin{equation} \label{Efree1}
  {\rm Re} \mathcal{E}_{free}(\widehat{u})(t,k)\lesssim |\widehat{u}|_2^2
\end{equation}
and
\begin{equation}\label{macro-Fourier}
  \begin{aligned}
    &\partial_t {\rm Re} \mathcal{E}_{free}(\widehat{u}(t,k))
    +\frac{\lambda|k|^2}{1+|k|^2}\left(|\widehat{a}|^2+|\widehat{b}|^2+|\widehat{c}|^2\right)\\[2mm]
    \lesssim ~~&{\epsilon}^2\left(|\widehat{b}|^2+|\widehat{c}|^2\right)+|\{{\bf{I}}-{\bf{P}}\}\widehat{u}|_{\nu}^2
    +|\nu^{-1/2}\widehat{G}|_2^2
  \end{aligned}
\end{equation}
for any $t\geq0$ and $k\in\mathbb{R}^3$.
\end{lemma}

\begin{proof}
\textit{Estimate on $\widehat{b}$. }
We claim that for $0<\eta<1$, it holds that
\begin{equation}\label{b-F-E}
  \begin{aligned}
    &\partial_t{\rm Re}\sum_{j}\left(\frac{1}{2}\sum_{m\neq j}ik_jA_{mm}(\{{\bf{I}}-{\bf{P}}\}\widehat{u})-
    \sum_{m}ik_mA_{jm}(\{{\bf{I}}-{\bf{P}}\}\widehat{u})|\widehat{b}_j\right)\\[2mm]
    &+(1-\eta)|k|^2|\widehat{b}|^2\\[2mm]
    \leq &\eta|k|^2\left(|\widehat{a}|^2+|\widehat{c}|^2\right)+{\epsilon}^2|\widehat{b}_j|^2
    +C_{\eta}(1+|k|^2)\left(|\{{\bf{I}}-{\bf{P}}\}\widehat{u}|_{\nu}^2
    +|\nu^{-1/2}\widehat{G}|_{2}^2\right).
  \end{aligned}
\end{equation}
In fact, the Fourier transform of $(\ref{elliptic})$ gives
\begin{equation*}
  \begin{aligned}
    &\partial_t\left\{\frac{1}{2}\sum_{m\neq j}ik_jA_{mm}(\{{\bf{I}}-{\bf{P}}\}\widehat{u})-
    \sum_{m}ik_m A_{jm}(\{{\bf{I}}-{\bf{P}}\}\widehat{u})\right\}+|k|^2\widehat{b}_j+k_{j}^2\widehat{b}_j\\[2mm]
    =&\frac{1}{2}\sum_{m\neq j}ik_jA_{mm}(\hat{R}+\widehat{G})-\sum_{m}ik_mA_{jm}(\hat{R}+\widehat{G}),
  \end{aligned}
\end{equation*}
where
\begin{equation*}
  R=-\xi\cdot\nabla_x \{{\bf{I}}-{\bf{P}}\}u+{\epsilon}L_{FP}\{{\bf{I}}-{\bf{P}}\}u+L\{{\bf{I}}-{\bf{P}}\}u.
\end{equation*}
We then take the complex inner product with $\widehat{b}_j$ to find
\begin{equation}\label{b-F-estimate}
  \begin{aligned}
    &\partial_t\left(\frac{1}{2}\sum_{m\neq j}ik_jA_{mm}(\{{\bf{I}}-{\bf{P}}\}\widehat{u})-
    \sum_{m}ik_mA_{jm}(\{{\bf{I}}-{\bf{P}}\}\widehat{u})|\widehat{b}_j\right)
    +\left(|k|^2+k_j^2\right)|\widehat{b}_j|^2\\[2mm]
    =&\left(\frac{1}{2}\sum_{m\neq j}ik_jA_{mm}(\hat{R}+\widehat{G})-\sum_{m}ik_mA_{jm}(\hat{R}+\widehat{G})|\widehat{b}_j\right)\\[2mm]
    &+\left(\frac{1}{2}\sum_{m\neq j}ik_jA_{mm}(\{{\bf{I}}-{\bf{P}}\}\widehat{u})-
    \sum_{m}ik_mA_{jm}(\{{\bf{I}}-{\bf{P}}\}\widehat{u})|\partial_t\widehat{b}_j\right)\\[2mm]
    =& I_1+I_2.
  \end{aligned}
\end{equation}
Note that
\begin{equation*}
  \hat{R}=-i\xi\cdot k\{{\bf{I}}-{\bf{P}}\}\widehat{u}
  +{\epsilon}L_{FP}\{{\bf{I}}-{\bf{P}}\}\widehat{u}+L\{{\bf{I}}-{\bf{P}}\}\widehat{u},
\end{equation*}
which implies
\begin{equation*}
  |A_{jm}(\hat{R})|^2\lesssim (1+|k|^2)|\{{\bf{I}}-{\bf{P}}\}\widehat{u}|_{\nu}^2.
\end{equation*}
Thus, $I_1$ is bounded by
\begin{equation} \label{I1}
\begin{aligned}
  I_1\leq&~~\eta|k|^2|\widehat{b}_j|^2+C_{\eta}\sum_{j,m}\left(|A_{jm}(\hat{R})|^2+|A_{jm}(\widehat{G})|^2\right)\\
  \leq &~~\eta|k|^2|\widehat{b}_j|^2+
  C_{\eta}(1+|k|^2)\left(|\{{\bf{I}}-{\bf{P}}\}\widehat{u}|_{\nu}^2+|\nu^{-1/2}\widehat{G}|_{2}^2\right).
\end{aligned}
\end{equation}
For $I_2$, using the Fourier transform of $(\ref{macro1})_2$
\begin{equation}\label{b's Fourier}
  \partial_t\widehat{b}_j+ik_j(\widehat{a}+2\widehat{c})
  +\sum_mik_mA_{jm}(\{{\bf{I}}-{\bf{P}}\}\widehat{u})+{\epsilon}\widehat{b}_j=0
\end{equation}
to replace $\partial_t \widehat{b}_j$, we have
\begin{equation} \label{I2}
  \begin{aligned}
    I_2
  \leq&~~\eta|k|^2\left(|\widehat{a}|^2+|\widehat{c}|^2\right)+{\epsilon}^2|\widehat{b}_j|^2
  +C_{\eta}(1+|k|^2)\sum_{jm}|A_{j,m}\{{\bf{I}}-{\bf{P}}\}\widehat{u}|_{2}^2\\
  \leq&~~\eta|k|^2\left(|\widehat{a}|^2+|\widehat{c}|^2\right)+{\epsilon}^2|\widehat{b}_j|^2
  +C_{\eta}(1+|k|^2)|\{{\bf{I}}-{\bf{P}}\}\widehat{u}|_{\nu}^2.
  \end{aligned}
\end{equation}
Therefore, one can take the real part of $(\ref{b-F-estimate})$ and plug the estimates (\ref{I1}) and (\ref{I2})
into it to discover (\ref{b-F-estimate}).

\textit{Estimate on $\widehat{c}$. }
For any $0<\eta<1$, we have
\begin{equation}\label{c-F-E}
  \begin{aligned}
    &\partial_t{\rm Re} \sum_{j}\left(B_j(\{{\bf{I}}-{\bf{P}}\}\widehat{u})|ik_j\widehat{c}\right)+(1-\eta)|k|^2|\widehat{c}|^2\\
    \leq&~~\eta|k|^2||\widehat{b}_j|^2+{\epsilon}^2|\widehat{c}|^2+C_{\eta}(1+|k|^2)\left(|\{{\bf{I}}-{\bf{P}}\}\widehat{u}|_{\nu}^2
    +|\nu^{-1/2}\widehat{G}|_{2}^2\right).
  \end{aligned}
\end{equation}
In fact, multiply the Fourier transform of $(\ref{macro2})_3$
\begin{equation*}
  \partial_tB_j(\{{\bf{I}}-{\bf{P}}\}\widehat{u})+ik_j\widehat{c}=B_j(\hat{R}+\widehat{G})
\end{equation*}
by $-ik_j \widehat{c}$ to give
\begin{equation*}
\begin{aligned}
  &\partial_t\left(B_j(\{{\bf{I}}-{\bf{P}}\}\widehat{u})|ik_j\widehat{c}\right)+|k_j|^2|\widehat{c}|^2\\[2mm]
  =&\left(B_j(\hat{R}+\widehat{G})|ik_j\widehat{c}\right)+\left(B_j(\{{\bf{I}}-{\bf{P}}\}\widehat{u})|ik_j\partial_t\widehat{c}\right)\\[2mm]
  =&~~I_3+I_4.
\end{aligned}
\end{equation*}
$I_3$ is bounded by
\begin{equation} \label{I3}
\begin{aligned}
  I_3\leq&\eta|k_j|^2|\widehat{c}_j|^2+C_{\eta}\sum_j\left(|B_j(\hat{R})|^2+|B_j(\widehat{G})|^2\right)\\
  \leq &\eta|k_j|^2|\widehat{c}_j|^2+C_{\eta}(1+|k|^2)\left(|\{{\bf{I}}-{\bf{P}}\}\widehat{u}|_{\nu}^2
    +|\nu^{-1/2}\widehat{G}|_{2}^2\right).
  \end{aligned}
\end{equation}
For $I_4$, using the Fourier transform of $(\ref{macro1})_3$
\begin{equation*}
  \partial_t\widehat{c}+\frac{1}{3}ik\cdot\widehat{b}+\frac{5}{3}\sum_jik_j B_j(\{{\bf I}-{\bf P}\}\widehat{u})+2{\epsilon}\widehat{c}=0
\end{equation*}
to replace $\partial_t \widehat{c}$, one has
\begin{equation} \label{I4}
  \begin{aligned}
    I_4
    \leq&\eta|k|^2||\widehat{b}_j|^2+{\epsilon}^2|\widehat{c}|^2
    +C_{\eta}(1+|k|^2)\sum_{j}|B_j(\{{\bf{I}}-{\bf{P}}\}\widehat{u})|^2\\[2mm]
    \leq &\eta|k|^2||\widehat{b}_j|^2+{\epsilon}^2|\widehat{c}|^2
    +C_{\eta}(1+|k|^2) |\{{\bf{I}}-{\bf{P}}\}\widehat{u}|_{\nu}^2.
  \end{aligned}
\end{equation}
Hence, (\ref{c-F-E}) follows by taking the real part and applying the estimates of (\ref{I3}) and (\ref{I4}),
and then taking the summation over $1\leq j\leq3$.

\textit{Estimate on $\widehat{a}$. } We claim that it holds for any $0\leq \eta<1$ that
\begin{equation}\label{a-F-E}
  \begin{aligned}
    &\partial_t{\rm Re}\sum_{j}\left(b_j|ik_j\widehat{a}\right)+(1-\eta)|k|^2|\widehat{a}|^2\\
    \leq&~|k|^2|\widehat{b}|^2+C_{\eta}\left(|k|^2|\widehat{c}|^2
  +|k|^2|\{{\bf{I}}-{\bf{P}}\}\widehat{u}|_{\nu}^2+{\epsilon}^2|\widehat{b}|^2\right).
  \end{aligned}
\end{equation}
In fact, using $(\ref{b's Fourier})$, and taking the complex inner product
with $ik_j\widehat{a}$, and then taking the summation over $1\leq j\leq3$,
one has
\begin{equation}\label{a-F-estimate}
  \begin{aligned}
    \partial_t\sum_{j}\left(\widehat{b}_j|ik_j\widehat{a}\right)+|k|^2|\widehat{a}|^2
    =&\sum_{j}\left(-2ik_j\widehat{c}|ik_j\widehat{a}\right)
    -\sum_{j,m}\left(ik_mA_{jm}(\{{\bf{I}}-{\bf{P}}\}\widehat{u})|ik_j\widehat{a}\right)\\[2mm]
    &+\sum_{j}\left(-{\epsilon}\widehat{b}_j|ik_j\widehat{a}\right)
    +\sum_{j}\left(\widehat{b}_j|ik_j\partial_t\widehat{a}\right).
  \end{aligned}
\end{equation}
The first there terms on the right-hand side of $(\ref{a-F-estimate})$ are bounded by
\begin{equation*}
  \eta|k|^2|\widehat{a}|^2+C_{\eta}\left(|k|^2|\widehat{c}|^2
  +|k|^2|\{{\bf{I}}-{\bf{P}}\}\widehat{u}|_{\nu}^2+{\epsilon}^2|\widehat{b}|^2\right),
\end{equation*}
while for the last term, it holds that
\begin{equation*}
  \sum_{j}\left(\widehat{b}_j|ik_j\partial_t\widehat{a}\right)=
  \sum_{j}\left(\widehat{b}_j|ik_j(-ik\cdot\widehat{b})\right)=|k\cdot\widehat{b}|^2\leq|k|^2|\widehat{b}|^2.
\end{equation*}
Here we used the Fourier transform of $(\ref{macro1})_1$:
\begin{equation*}
   \partial_t\widehat{a}+ik\cdot\widehat{b}=0.
\end{equation*}
Then, one can deduce (\ref{a-F-E}) by putting the above estimates into $(\ref{a-F-estimate})$ and taking the real part.

Therefore, $(\ref{macro-Fourier})$ follows
from the proper linear combination of (\ref{b-F-E}), (\ref{c-F-E}) and (\ref{a-F-E})
by taking $M>0$ large enough and $0<\eta<1$ small enough.
Note that
\begin{equation*}
  \begin{aligned}
     |\mathcal{E}_{free}(\widehat{u})|(t,k)\lesssim&\left(|\widehat{a}|^2+|\widehat{b}|^2+|\widehat{c}|^2\right)
     +\sum_{j,m}\left(|A_{jm}(\{{\bf{I}}-{\bf{P}}\}\widehat{u})|^2+|B_j(\{{\bf I}-{\bf{P}}\}\widehat{u})|^2\right)\\[2mm]
     \lesssim&|{\bf{P}}\widehat{u}|_{2}^2+|\{{\bf{I}}-{\bf{P}}\}\widehat{u}|_{2}^2\lesssim |\widehat{u}|_2^2.
  \end{aligned}
\end{equation*}
This completes the proof of lemma $\ref{macro-thm}$.
\end{proof}

\begin{lemma}
  $\kappa_1>0$ exists such that $\mathcal{E}(\widehat{u})(t,k)$, which is defined by
  \begin{equation}
    \label{Ecal}
    \mathcal{E}(\widehat{u})=|\widehat{u}|_2^2+\kappa_1{\rm Re}~\mathcal{E}_{free}(\widehat{u}),
  \end{equation}
  satisfies that
  \begin{equation} \label{Eest0}
    \mathcal{E}(\widehat{u})\sim|\widehat{u}|_2^2
  \end{equation}
  and
  \begin{equation}\label{E's decay}
    \mathcal{E}(\widehat{u})(t,k)
    \leq \mathcal{E}(\widehat{u})(0,k)e^{-\frac{\lambda|k|^2}{1+|k|^2}t}
    +C\int_0^t e^{-\frac{\lambda|k|^2}{1+|k|^2}(t-s)}|\nu^{-1/2}\widehat{G}(s,k)|_{2}^2ds
  \end{equation}
for any $t\geq0$ and $k\in\mathbb{R}^3$.
\end{lemma}
\begin{proof}
We first claim that for any $t\geq0$ and $k\in\mathbb{R}^3$, it holds that
\begin{equation}\label{micro-Fourier}
  \partial_t|\widehat{u}|_2^2+\kappa\left\{|\{{\bf{I}}-{\bf{P}}\}\widehat{u}|_{\nu}^2
  +{\epsilon}|\{{\bf{I}}-{\bf{P}}_0\}\widehat{u}|_{2}^2\right\}
  \lesssim |\nu^{-1/2}\widehat{G}|_2^2.
\end{equation}
In fact,
the Fourier transform of $(\ref{linear-equation})$ gives
\begin{equation}\label{u's Fouier}
  \partial_t\widehat{u}+i\xi\cdot k\widehat{u}=L\widehat{u}+{\epsilon}L_{FP}\widehat{u}+\widehat{G}.
\end{equation}
Further, taking the complex inner product with $\widehat{u}$ and taking the real part yield
\begin{equation}
  \label{mFourier1}
  \begin{aligned}
    &\frac{1}{2}\partial_t|\widehat{u}|_2^2-{\rm Re}\int_{\mathbb{R}^3}\left(L\widehat{u}|\widehat{u}\right)d\xi\\[2mm]
    =&{\epsilon}{\rm Re}\int_{\mathbb{R}^3}\left(L_{FP}\widehat{u}|\widehat{u}\right)d\xi
    +{\rm Re}\int_{\mathbb{R}^3}\left(\widehat{G}|\widehat{u}\right)d\xi.
  \end{aligned}
\end{equation}
For the second term on the left hand side of (\ref{mFourier1}), we have from (\ref{L0}) that
\begin{equation*}
  -{\rm Re}\int_{\mathbb{R}^3}\left(L\widehat{u}|\widehat{u}\right)d\xi\geq
  \frac{1}{2}\lambda_0|\{{\bf{I}}-{\bf{P}}\}\widehat{u}|_{\nu}^2.
\end{equation*}
For the two terms on the right-hand side of (\ref{mFourier1}), we have
\begin{equation*}
  {\epsilon}{\rm Re}\int_{\mathbb{R}^3}\left(L_{FP}\widehat{u}|\widehat{u}\right)d\xi
  \leq -\frac{1}{2}{\epsilon}\lambda_{FP}|\{{\bf{I}}-{\bf{P}}_0\}\widehat{u}|_{2}^2
\end{equation*}
and
\begin{equation*}
\begin{aligned}
  {\rm Re}\int_{\mathbb{R}^3}\left(\widehat{G}|\widehat{u}\right)d\xi
  =&{\rm Re}\int_{\mathbb{R}^3}\left(\widehat{G}|\{{\bf{I}}-{\bf{P}}\}\widehat{u}\right)d\xi
  \\[2mm] \leq&
  \frac{1}{4}\lambda_0|\{{\bf{I}}-{\bf{P}}\}\widehat{u}|_{\nu}^2+C |\nu^{-1/2}\widehat{G}|_2^2.
\end{aligned}
\end{equation*}
Here we used ${\bf P}h=0$. Plugging the above estimates into (\ref{mFourier1}) yields (\ref{micro-Fourier}).
Note that $|\widehat{b}|^2+|\widehat{c}|^2\lesssim |\{{\bf{I}}-{\bf{P}}_0\}\widehat{u}|_{2}^2$.
By taking $\kappa_1>0$ small enough,
it follows from (\ref{macro-Fourier}) and (\ref{micro-Fourier}) that
\begin{equation} \label{micro1}
    \partial_t\mathcal{E}(\widehat{u})(t,k)+\frac{\lambda|k|^2}{1+|k|^2}|{\bf P}\widehat{u}|^2
    +\lambda|\{{\bf{I}}-{\bf{P}}\}\widehat{u}|_{\nu}^2
    \lesssim |\nu^{-1/2}\widehat{G}|_{2}^2.
\end{equation}
(\ref{Efree1}) implies (\ref{Eest0}) by further taking $\kappa_1>0$ small enough.
Here, we consider the hard potential case, i.e., $0\leq \gamma\leq 1$.
Thus, we have
\begin{equation}
  \label{micro2}
  \mathcal{E}(\widehat{u})(t,k)\lesssim
  |\widehat{u}|^2_2\lesssim |{\bf P}\widehat{u}|^2+|\{{\bf{I}}-{\bf{P}}\}\widehat{u}|_{\nu}^2.
\end{equation}
Pplug (\ref{micro2}) into (\ref{micro1}) to find
  \begin{equation} \label{Eest1}
    \partial_t\mathcal{E}(\widehat{u})(t,k)+\frac{\lambda|k|^2}{1+|k|^2}\mathcal{E}(\widehat{u})(t,k)
    \lesssim |\nu^{-1/2}\widehat{G}|_{2}^2
\end{equation}
which by the Gronwall's inequality, implies
(\ref{E's decay}).
This completes the proof of Lemma 4.2.
\end{proof}

Now, to prove ,
let $h=0$ so that $u_1(t)={e^{tB}}u_0$ is the solution
to the Cauchy problem $(\ref{linear-equation})$ and
hence satisfies the estimate (\ref{E's decay}) with $h = 0$:
\begin{equation} \label{u_10}
  \mathcal{E}(\widehat{u_1})(t,k)
    \leq \mathcal{E}(\widehat{u_1})(0,k)e^{-\frac{\lambda|k|^2}{1+|k|^2}t}.
\end{equation}
Write $k^{\alpha}=k_1^{\alpha_1}k_2^{\alpha_2}k_3^{\alpha_3}$.
Paseval's identity and $(\ref{Eest0})$ yield
\begin{equation} \label{43}
     \left\|\partial^{\alpha}_xu_1\right\|^2
     \lesssim \int_{\mathbb{R}^3}|k^{2\alpha}||\widehat{u_1}(t,k)|_{^2}^2dk
     \lesssim \int_{\mathbb{R}^3}|k^{2\alpha}|\mathcal{E}(\widehat{u_1})(t,k)dk.
\end{equation}
Then, from (\ref{u_10}) and (\ref{Eest0}), one has
\begin{equation} \label{42}
  \|\partial^{\alpha}_xu_1\|^2\lesssim \int_{\mathbb{R}^3}
  |k^{2\alpha}|e^{-\frac{\lambda|k|^2}{1+|k|^2}t}|\widehat{u_0}|_{2}^2dk.
\end{equation}
As in $\cite{Kawashima}$, one can further estimate (\ref{42}) as
\begin{equation}\label{part-1-decay}
   \begin{aligned}
     \|\partial^{\alpha}_xu_1\|^2
     \lesssim&\int_{|k|\leq1}|k^{2\alpha}|e^{-\frac{\lambda|k|^2}{1+|k|^2}t}|\widehat{u_0}|_{2}^2dk
     +\int_{|k|\geq1}|k^{2\alpha}|e^{-\frac{\lambda|k|^2}{1+|k|^2}t}|\widehat{u_0}|_{2}^2dk\\[2mm]
     \lesssim&\int_{|k|\leq1}|k^{2\alpha}|e^{-\frac{\lambda|k|^2}{1+|k|^2}t}dk
     \|u_0\|_{Z_1}^2+e^{-\frac{\lambda}{2}t}\|\partial^{\alpha}_xu_0\|^2\\[2mm]
     \lesssim&(1+t)^{-\frac{3}{2}-|\alpha|}\left(\|u_0\|_{Z_1}^2+\|\partial^{\alpha}_xu_0\|^2\right).
   \end{aligned}
\end{equation}
Here, we used the Hausdorff-Young inequality $$\sup_{|k|\leq 1}|\widehat{u_0}(k,\xi)|
\lesssim \int_{\mathbb{R}^3}|u_0|(x,\xi)dx.$$
Next, let $u_0=0$ so that
\begin{equation*}
  u_2(t)=\int_0^t{e^{(t-s)B}}G(s)ds
\end{equation*}
is the solution of the Cauchy problem $(\ref{linear-equation})$ with $u_0=0$.
Then, similar to (\ref{43}) and (\ref{part-1-decay}), one has
\begin{equation}\label{part-2-decay}
   \begin{aligned}
     &\left\|\partial^{\alpha}_x\int_0^te^{(t-\tau)B}G(s)ds\right\|^2\\[2mm]
     \lesssim&\int_0^t\int_{\mathbb{R}^3}|k^{2\alpha}|e^{-\frac{|k|^2}{1+|k|^2}(t-s)}
     |\nu^{-1/2}\widehat{G}(s)|_{2}^2dkds\\[2mm]
     \lesssim&\int_0^t(1+t-s)^{-\frac{3}{2}-|\alpha|}\left(\|\nu^{-1/2}G(s)\|_{Z_1}^2
     +\|\nu^{-1/2}\partial^{\alpha}_xG(s)\|^2\right)ds.
   \end{aligned}
\end{equation}
Recall that the solution $u$ to the Cauchy problem $(\ref{u})$-$(\ref{u_0})$
can be formally written as
\begin{equation*}
  u(t)={e^{tB}}u_0+\int_0^t{e^{(t-s)B}}\Gamma(u,u)(s)ds.
\end{equation*}
Thus, $(\ref{part-1-decay})$ and $(\ref{part-2-decay})$ yield
\begin{equation} \label{44}
   \begin{aligned}
     \|u\|^2
     \lesssim&(1+t)^{-\frac{3}{2}}\left(\|u_0\|_{Z_1}^2+\|u_0\|^2\right)\\[2mm]
     &+\int_0^t(1+t-s)^{-\frac{3}{2}}\Big(\|\nu^{-1/2}\Gamma(u,u)(s)\|_{Z_1}^2
     +\|\nu^{-1/2}\Gamma(u,u)(s)\|^2\Big)ds.
   \end{aligned}
\end{equation}
In the following, we shall estimate the terms on the right hand side of (\ref{44}).
For this, we first note that for $0\leq \gamma\leq 1$,
\begin{equation} \label{UY}
\begin{aligned}
    |\nu^{-1/2}\Gamma(u,u)|_2\lesssim &|\nu^{1/2}u|_2|u|_2,\\[2mm]
    \|\nu^{-1/2}\Gamma(u,u)\|_{Z_1}\lesssim& \|\nu^{1/2}u\|\|u\|,
\end{aligned}
\end{equation}
which are proved in \cite{Golse}  and \cite{Ukai_Yang}, respectively.
Thus, one can discover from (\ref{UY}) that if $\gamma\leq 2 l(q-\gamma)$,
then it holds that
\begin{equation}\label{nu-gamma}
\begin{aligned}
  &\|\nu^{-1/2}\Gamma(u,u)(s)\|_{Z_1}
  +\|\nu^{-1/2}\Gamma(u,u)(s)\|\\[2mm]
  \lesssim &
  \|\nu^{1/2}u\|\left(\|u\|+\sup_{x}|u|_2\right)
  \lesssim \mathcal{E}_{q,l}(u)(t).
\end{aligned}
\end{equation}
For $t\geq0$, define a temporal function by
\begin{equation}\label{XNl}
  X_{q,l}(u)(t)
  =\sup_{0\leq s\leq t}(1+s)^{\frac{3}{2}}\mathcal{E}_{q,l}(u)(s).
\end{equation}
Hence, it follows from (\ref{44}) and (\ref{nu-gamma}) that
\begin{equation}
  \label{45}
  \begin{aligned}
    \|u\|^2\lesssim & (1+t)^{-\frac{3}{2}}\left(\|u_0\|_{Z_1}^2+\|u_0\|^2\right)\\[2mm]
     &+\int_0^t(1+t-s)^{-\frac{3}{2}}(1+s)^{-3}X_{q,l}(u)^2(s)ds\\[2mm]
     \lesssim & (1+t)^{-\frac{3}{2}}\left(\|u_0\|_{Z_1}^2+\|u_0\|^2
     +X_{q,l}(u)^2(t)\right).
  \end{aligned}
\end{equation}
Here we used $X_{N,l}(t)$ is nondecreasing in $t$ and
\begin{equation*}
  \int_0^t(1+t-s)^{-\frac{3}{2}}(1+s)^{-3}ds\lesssim (1+t)^{-\frac{3}{2}}.
\end{equation*}
By comparing (\ref{E}) and (\ref{D}), it holds that
\begin{equation*}
  \mathcal{D}_{q,l}(u)(t)+\|(a,b,c)(t)\|^2\geq \kappa\mathcal{E}_{q,l}(u)(t).
\end{equation*}
Then it follows from $(\ref{Lyapunov})$ that
  \begin{equation}\label{new-Lyapunov}
  \frac{d}{dt}\mathcal{E}_{q,l}(u)(t)+\kappa\mathcal{E}_{q,l}(u)(t)\lesssim \|(a,b,c)(t)\|^2\lesssim \|u(t)\|^2.
  \end{equation}
Due to the Gronwall inequality, (\ref{new-Lyapunov}) together with (\ref{45}) imply
\begin{equation*}
\begin{aligned}
  \mathcal{E}_{q,l}(u)(t)
  \lesssim &\mathcal{E}_{q,l}(u_0)e^{-\kappa t}
             +\int_0^t e^{-\kappa(t-s)}\|u(s)\|^2ds\\[2mm]
  \lesssim & (1+t)^{-\frac{3}{2}}\left(\|u_0\|_{Z_1}^2+\mathcal{E}_{q,l}(u_0)
     +X_{q,l}(u)^2(t)\right),
\end{aligned}
\end{equation*}
which implies
\begin{equation*}
  X_{q,l}(u)(t)\lesssim \|u_0\|_{Z_1}^2+\mathcal{E}_{q,l}(u_0)
  +X_{q,l}(u)^2(t).
\end{equation*}
This proves the decay rate stated in our Theorem
for the hard potential case, i.e., $0\leq\gamma\leq 1$ with the help of Strauss' Lemma.

\subsection{The soft potential case}
In this subsection, we shall obtain the time decay of the solution $u$ to the Cauchy problem
(\ref{u})-(\ref{u_0}) in the soft potential case ($-3<\gamma<0$).
For this, we first establish the time decay of the evolution operator ${e^{tB}},$
which is stated as follows.
\begin{lemma}\label{weight-decay}
Define $\mu=\mu(\xi)=\langle\xi\rangle^{-\frac{\gamma}{2}}$.
Let $-3<\gamma<0$, $l\geq0$ and $l_0>\frac{3}{2}$. If
\begin{equation*}
  \left\|\mu^{l+l_0}u_0\right\|_{Z_1}+\left\|\mu^{l+l_0}u_0\right\|<\infty,
\end{equation*}
then the evolution operator ${e^{tB}}$ satisfies
\begin{equation}\label{soft-0-decay}
  \left\|\mu^l{e^{tB}}u_0\right\|
  \lesssim (1+t)^{-\frac{3}{4}}\left(\left\|\mu^{l+l_0}u_0\right\|_{Z_1}+\left\|\mu^{l+l_0}u_0\right\|\right)
\end{equation}
for each $t\geq 0$.
\end{lemma}

\begin{proof}
Let $G=0$ so that $u_1(t)={e^{tB}}u_0$ is the solution
to the Cauchy problem $(\ref{linear-equation})$.
Apply $\{{\bf I}-{\bf P}\}$ to (\ref{u's Fouier}) with $G=0$ to find
\begin{equation*}
  \partial_t\{{\bf{I}}-{\bf{P}}\}\widehat{u_1}+i\xi\cdot k\{{\bf{I}}-{\bf{P}}\}\widehat{u_1}
  =L\{{\bf{I}}-{\bf{P}}\}\widehat{u_1}+{\epsilon}L_{FP}\{{\bf{I}}-{\bf{P}}\}\widehat{u_1}
  +{\bf{P}}i\xi\cdot k\widehat{u_1}-i\xi\cdot k{\bf{P}}\widehat{u_1}.
\end{equation*}
By further taking the complex inner product of the above equation with
$\mu^{2l}\{{\bf{I}}-{\bf{P}}\}\widehat{u_1}$ and integrating it over $\mathbb{R}_{\xi}^3$,
we have
\begin{equation}
  \label{soft0}
  \begin{aligned}
    &\partial_t\left|\mu^l\{{\bf{I}}-{\bf{P}}\}\widehat{u_1}\right|_2^2
    +\kappa\left|\mu^l\{{\bf{I}}-{\bf{P}}\}\widehat{u_1}\right|_{\nu}^2\\[2mm]
    \lesssim &
    \left|\{{\bf{I}}-{\bf{P}}\}\widehat{u_1}\right|_{\nu}^2
    +{\rm Re}\int_{\mathbb{R}^3}
    \left({\bf{P}}i\xi\cdot k\widehat{u_1}-i\xi\cdot k{\bf{P}}\widehat{u_1}|
    \mu^{2l}\{{\bf{I}}-{\bf{P}}\}\widehat{u_1}\right)d\xi,
  \end{aligned}
\end{equation}
whenever $(q-\gamma)^2{\epsilon}$ is small enough.
Here, we used (\ref{L1}) and (\ref{Lfp w1}).
For the second term on the right hand side of (\ref{soft0}), it holds that
for each $|k|\leq 1$,
\begin{equation*}
  \begin{aligned}
    &{\rm Re}\int_{\mathbb{R}^3}
    \left({\bf{P}}i\xi\cdot k\widehat{u_1}-i\xi\cdot k{\bf{P}}\widehat{u_1}|
    \mu^{2l}\{{\bf{I}}-{\bf{P}}\}\widehat{u_1}\right)d\xi\\[2mm]
    \lesssim & \left|\{{\bf{I}}-{\bf{P}}\}\widehat{u_1}\right|_{\nu}^2
    +|k|^2\left(|{\bf P}\widehat{u_1}|_2^2+
    \left|\{{\bf{I}}-{\bf{P}}\}\widehat{u_1}\right|_{\nu}^2\right)\\[2mm]
    \lesssim & \left|\{{\bf{I}}-{\bf{P}}\}\widehat{u_1}\right|_{\nu}^2
    +\frac{|k|^2}{1+|k|^2}|{\bf P}\widehat{u_1}|_2^2.
  \end{aligned}
\end{equation*}
Thus, using $\mu=\langle\xi\rangle^{-\frac{\gamma}{2}}$, we get
\begin{equation}
  \label{soft1}
  \begin{aligned}
    &\partial_t\left|\mu^l\{{\bf{I}}-{\bf{P}}\}\widehat{u_1}\right|_2^2\chi_{|k|\leq1}
    +\kappa\left|\mu^{l-1}\{{\bf{I}}-{\bf{P}}\}\widehat{u_1}\right|_2^2\chi_{|k|\leq1}\\[2mm]
    \lesssim &
    \left|\{{\bf{I}}-{\bf{P}}\}\widehat{u_1}\right|_{\nu}^2
    +\frac{|k|^2}{1+|k|^2}|{\bf P}\widehat{u_1}|_2^2.
  \end{aligned}
\end{equation}
To obtain the velocity-weighted estimate for the pointwise time-frequency variables
over $|k| \geq 1$, we directly take the complex inner product of (\ref{u's Fouier}) with $G=0$
with $\mu^{2l}\widehat{u_1}$ and integrate in over $\mathbb{R}^3_{\xi}$ to discover
\begin{equation}
  \label{soft2}
    \partial_t\left|\mu^l\widehat{u_1}\right|_2^2
    +\kappa\left|\mu^{l-1}\widehat{u_1}\right|_{2}^2
    \lesssim
    \left|\widehat{u_1}\right|_{\nu}^2,
\end{equation}
whenever $(q-\gamma)^2{\epsilon}$ is small enough.
Note that $$\frac{|k|^2}{1+|k|^2}\chi_{|k|\geq 1}\geq\frac{1}{2}.$$
It follows that
\begin{equation}\label{soft3}
\begin{aligned}
    &\partial_t\left|\mu^l\widehat{u_1}\right|_2^2\chi_{|k|\geq1}
    +\kappa\left|\mu^{l-1}\widehat{u_1}\right|_{2}^2\chi_{|k|\geq1}\\[2mm]
    \lesssim &
    \left|\{{\bf{I}}-{\bf{P}}\}\widehat{u_1}\right|_{\nu}^2
    +\frac{|k|^2}{1+|k|^2}|{\bf P}\widehat{u_1}|_2^2.
\end{aligned}
\end{equation}
Therefore, for $\kappa_2>0$ small enough, a suitable linear combination of
(\ref{soft1}), (\ref{soft3}) and (\ref{micro1}) with $h\equiv 0$
\begin{equation}
    \partial_t\mathcal{E}(\widehat{u_1})(t,k)+\frac{\lambda|k|^2}{1+|k|^2}|{\bf P}\widehat{u_1}|_2^2
    +\lambda|\{{\bf{I}}-{\bf{P}}\}\widehat{u_1}|_{\nu}^2
    \leq 0
\end{equation}
 yields that
whenever $l \geq  0$,
\begin{equation} \label{soft x}
  \partial_tE_l(\widehat{u_1})+\kappa D_l(\widehat{u_1})\leq 0
\end{equation}
where $E_l(\widehat{u_1})$ and $D_l(\widehat{u_1})$ are given by
\begin{equation*}
  \begin{aligned}
    E_l(\widehat{u_1})=&\mathcal{E}(\widehat{u_1})
                        +\kappa_2\left(\left|\mu^l\{{\bf{I}}-{\bf{P}}\}\widehat{u_1}\right|_{2}^2\chi_{|k|\leq 1}
                        +\left|\mu^l\widehat{u_1}\right|_{2}^2\chi_{|k|\geq 1}\right),\\[2mm]
    D_l(\widehat{u_1})=&\left|\mu^{l-1}\{{\bf{I}}-{\bf{P}}\}\widehat{u_1}\right|_{2}^2
                        +\frac{|k|^2}{1+|k|^2}\left|{\bf P}\widehat{u_1}\right|_2^2.
  \end{aligned}
\end{equation*}
Due to $(\ref{Eest0})$ and the fact ${\bf P}\widehat{u_1}$
decays exponentially in $\xi$, it is clear that
\begin{equation} \label{equvi}
    E_l(\widehat{u_1}) \sim \left|{\bf P}\widehat{u_1}\right|_2^2
    +|\mu^l\{{\bf{I}}-{\bf{P}}\}\widehat{u_1}|_2^2
    \sim\left|\mu^l\widehat{u_1}\right|_2^2.
\end{equation}

Set
$$\rho(k)=\frac{|k|^2}{1+|k|^2}.$$
Let $0<\eta\leq1$ and $J> 0$ be chosen later.
Multiplying $(\ref{soft x})$ by $[1+\eta\rho(k)t]^J$,
we have from (\ref{equvi}) that
\begin{equation}\label{1-lya-weight}
\begin{aligned}
    &\partial_t\left\{[1+\eta\rho(k)t]^JE_l(\widehat{u_1})\right\}
    +\kappa [1+\eta\rho(k)t]^JD_{l}(\widehat{u_1})\\[2mm]
    \leq &J[1+\eta\rho(k)t]^{J-1}\eta\rho(k)E_l(\widehat{u_1})\\[2mm]
    \leq &CJ[1+\eta\rho(k)t]^{J-1}\eta\rho(k)\left|{\bf P}\widehat{u_1}\right|_2^2
    +CJ[1+\eta\rho(k)t]^{J-1}\eta\rho(k)|\mu^l\{{\bf{I}}-{\bf{P}}\}\widehat{u_1}|_2^2\\[2mm]
    \leq &\eta C[1+\eta\rho(k)t]^{J}D_l(\widehat{u_1})
    +CJ[1+\eta\rho(k)t]^{J-1}\eta\rho(k)|\mu^l\{{\bf{I}}-{\bf{P}}\}\widehat{u_1}|_2^2.
\end{aligned}
\end{equation}
In what follows, we estimate the second term on the right hand side of (\ref{1-lya-weight}).
To this end, let $p>1$ be chosen later. Then it holds that
\begin{equation} \label{soft5}
\begin{aligned}
  &\left|\mu^l\{{\bf{I}}-{\bf{P}}\}\widehat{u_1}\right|_2^2\\[2mm]
  \leq &\left|\mu^l\{{\bf{I}}-{\bf{P}}\}\widehat{u_1}\chi_{\mu^2(\xi)\leq [1+\eta\rho(k)t]}\right|_2^2
        +\left|\mu^l\{{\bf{I}}-{\bf{P}}\}\widehat{u_1}\chi_{\mu^2(\xi)>[1+\epsilon\rho(k)t]}\right|_2^2\\[2mm]
  \leq &[1+\eta\rho(k)t]\left|\mu^{l-1}\{{\bf{I}}-{\bf{P}}\}\widehat{u_1}\right|_2^2
        +[1+\eta\rho(k)t]^{-p-J+1}\left|\mu^{l+p+J-1}\{{\bf{I}}-{\bf{P}}\}\widehat{u_1}\right|_2^2\\[2mm]
  \leq &[1+\eta\rho(k)t]D_{l}(\widehat{u_1})
        +C[1+\eta\rho(k)t]^{-p-J+1}E_{l+p+J-1}(\widehat{u_1}).
\end{aligned}
\end{equation}
Here, we used the splitting
\begin{equation*}
    1=\chi_{\mu^2(\xi)\leq [1+\eta\rho(k)t]}+\chi_{\mu^2(\xi)>[1+\eta\rho(k)t]}
\end{equation*}
and (\ref{equvi}).
Plugging (\ref{soft5}) into (\ref{1-lya-weight}) and noting that
$E_{l+p+J-1}(\widehat{u_1})\leq E_{l+p+J-1}(\widehat{u_0})$
from (\ref{soft x}) due to $l+p+J-1\geq 0$,
one has
\begin{equation*}
  \begin{aligned}
         &\partial_t\left\{[1+\eta\rho(k)t]^JE_l(\widehat{u_1})\right\}
           +\kappa [1+\eta\rho(k)t]^JD_{l}(\widehat{u_1})\\[2mm]
    \leq &\epsilon C[1+\eta\rho(k)t]^{J}D_l(\widehat{u_1})
           +C[1+\eta\rho(k)t]^{-p}\eta\rho(k)E_{l+p+J-1}(\widehat{u_0}),
  \end{aligned}
\end{equation*}
which implies
\begin{equation*}
         \partial_t\left\{[1+\eta\rho(k)t]^JE_l(\widehat{u_1})\right\}
           +\lambda [1+\eta\rho(k)t]^JD_{l}(\widehat{u_1})
    \lesssim [1+\eta\rho(k)t]^{-p}\eta\rho(k)E_{l+p+J-1}(\widehat{u_0}),
\end{equation*}
whenever $\eta>0$ is small enough. Integrating the above inequality, using
\begin{equation*}
  \int_0^t[1+\eta\rho(k)s]^{-p}\eta\rho(k)ds\leq \int_0^{\infty}[1+s]^{-p}ds< \infty
\end{equation*}
for $p>1$, and noting $p+J-1>0$, we have
\begin{equation*}
  [1+\eta\rho(k)t]^JE_l(\widehat{u_1})
  \lesssim E_l(\widehat{u_0}) + E_{l+p+J-1}(\widehat{u_0})
  \lesssim E_{l+p+J-1}(\widehat{u_0}).
\end{equation*}
Now, for any given $l_0>\frac{3}{2}$,
we choose $J>\frac{3}{2}$ and $p>1$ such that $p+J-1=l_0$ to get
\begin{equation} \label{soft6}
    E_l(\widehat{u_1})\lesssim [1+\eta\rho(k)t]^{-J}E_{l+l_0}(\widehat{u_0}).
\end{equation}
Since $J>\frac{3}{2}$, (\ref{equvi}), (\ref{soft6}) and Hausdorff-Young inequality yield
\begin{equation*}
  \begin{aligned}
    \left\|\mu^lu_1\right\|^2\lesssim&\int_{\mathbb{R}^3}\left|\mu^l\widehat{u_1}\right|_2^2dk
                             \lesssim\int_{\mathbb{R}^3}E_l(\widehat{u_1})dk\\[2mm]
                             \lesssim&\sup_{k}E_{l+l_0}(\widehat{u_0})
                                      \int_{|k|\leq 1}[1+\eta\rho(k)t]^{-J}dk
                                     +(1+ t)^{-J}\int_{|k|\geq 1}E_{l+l_0}(\widehat{u_0})dk\\[2mm]
                             \lesssim&(1+t)^{-\frac{3}{2}}\left(\left\|\mu^{l+l_0}u_0\right\|_{Z_1}^2
                                      +\left\|\mu^{l+l_0}u_0\right\|^2\right).
  \end{aligned}
\end{equation*}
This completes the proof.
\end{proof}
Recall that the solution $u$ to the Cauchy problem $(\ref{u})$-$(\ref{u_0})$
can be formally written as
\begin{equation*}
  u(t)={e^{tB}}u_0+\int_0^t{e^{(t-s)B}}\Gamma(u,u)(s)ds.
\end{equation*}
Thus, one has
\begin{equation*}
  \begin{aligned}
    \|u\|^2\lesssim&(1+t)^{-\frac{3}{2}}\left(\|\mu^{l_0}u_0\|_{Z_1}^2+\|\mu^{l_0}u_0\|^2\right)\\[2mm]
    &+\int_0^t(1+t-s)^{-\frac{3}{2}}\Big(\|\mu^{l_0}\Gamma(u,u)(s)\|_{Z_1}^2
    +\|\mu^{l_0}\Gamma(u,u)(s)\|^2\Big)ds,
  \end{aligned}
\end{equation*}
from Lemma 4.3. To estimate the time integral term on the right hand side of the above inequality,
we note that
\begin{equation*}
  \|\mu^{l_0}\Gamma(u,u)(t)\|_{Z_1}+\|\mu^{l_0}\Gamma(u,u)(t)\|
  \lesssim \sum_{|\alpha|+|\beta|\leq N}\|\partial_{\beta}^{\alpha}u\|
           \sum_{|\alpha|\leq N}\|\langle\xi\rangle^{\max\{0,\frac{\gamma}{2}(1-l_0)\}}\partial_{x}^{\alpha}u\|,
\end{equation*}
which is proved in \cite{Duan_Yang_Zhao2}.
Then we have
\begin{equation*}
  \|\mu^{l_0}\Gamma(u,u)(t)\|_{Z_1}+\|\mu^{l_0}\Gamma(u,u)(t)\|\lesssim \mathcal E_{q,l-1}(u)(t),
\end{equation*}
whenever $\frac{\gamma}{2}(1-l_0)\leq (q-\gamma)(l-1)$.
Moreover, it follows that
\begin{equation}  \label{soft7}
    \|u\|^2\lesssim
    (1+t)^{-\frac{3}{2}}\left(\|\mu^{l_0}u_0\|_{Z_1}^2+\|\mu^{l_0}u_0\|^2+X_{q,l-1}(u)^2(t)\right).
\end{equation}

Let $0<\eta<1/2$.
Notice that $(\ref{Lyapunov})$ also holds when $l$ is replaced by $l-1$
under the assumption $l\geq N+1$, $\sup_{0\leq t\leq T}\mathcal{E}_{q,l}(s)\leq \delta_0$
and $(q-\gamma)^2{\epsilon}\leq \delta_0$.
Thus, it holds that
\begin{equation*}
  \frac{d}{dt}{\mathcal{E}}_{q,l-1}(u)(t)+\lambda \mathcal{D}_{q,l-1}(u)(t)\leq 0.
\end{equation*}
Multiplying the above inequality by $(1+t)^{{3}/{2}+\eta}$ gives
\begin{equation}\label{soft-1-decay}
  \frac{d}{dt}\left\{(1+t)^{\frac{3}{2}+\eta}{\mathcal{E}}_{q,l-1}(u)(t)\right\}
  +\lambda(1+t)^{\frac{3}{2}+\eta}\mathcal{D}_{q,l-1}(u)(t)\lesssim (1+t)^{\frac{1}{2}+\eta}{\mathcal{E}}_{q,l-1}(u)(t).
\end{equation}
Similarly, from $(\ref{Lyapunov})$ with $l$ replaced by $l-\frac{1}{2}$ and further multiplying it by $(1+t)^{{1}/{2}+\eta}$, one has
\begin{equation}\label{soft-2-decay}
\begin{aligned}
  &\frac{d}{dt}\left\{(1+t)^{\frac{1}{2}+\eta}{\mathcal{E}}_{q,l-\frac{1}{2}}(u)(t)\right\}
  +\lambda(1+t)^{\frac{1}{2}+\eta}\mathcal{D}_{q,l-\frac{1}{2}}(u)(t)\\[2mm]
  \lesssim &(1+t)^{-\frac{1}{2}+\eta}{\mathcal{E}}_{q,l-\frac{1}{2}}(u)(t)
  \lesssim {\mathcal{E}}_{q,l-\frac{1}{2}}(u)(t).
\end{aligned}
\end{equation}
Note from $(\ref{E})$, $(\ref{D})$ that
\begin{equation*}
  {\mathcal{E}}_{q,l'-\frac{1}{2}}(u)(t)\lesssim \mathcal{D}_{q,l'}(u)(t)+\|{\bf P}u\|^2
\end{equation*}
holds for any given $l'$.
Then, from taking integration over $[0,t]$ of (\ref{soft-1-decay}),
(\ref{soft-2-decay}) and (\ref{Lyapunov}) and further taking
the appropriate linear combination, we have
\begin{equation*}
  (1+t)^{\frac{3}{2}+\eta}{\mathcal{E}}_{q,l-1}(u)(t)
  \lesssim {\mathcal{E}}_{q,l}(u_0) +\int_0^t(1+s)^{\frac{1}{2}+\eta}\|{\bf P}u(s)\|^2ds.
\end{equation*}
Thus, applying the estimate $(\ref{soft7})$ to the second term on the right hand side
of the above inequality and noticing
\begin{equation*}
  \int_0^t(1+s)^{\frac{1}{2}+\eta}(1+s)^{-\frac{3}{2}}ds\lesssim (1+t)^{\eta},
\end{equation*}
we have
\begin{equation*}
  (1+t)^{\frac{3}{2}+\eta}{\mathcal{E}}_{q,l-1}(u)(t)
  \lesssim {\mathcal{E}}_{q,l}(u_0)+(1+t)^{\eta}\left\{
  \|\mu^{l_0}u_0\|_{Z_1}^2+\|\mu^{l_0}u_0\|^2+X_{q,l-1}(u)^2(t)\right\},
\end{equation*}
which implies
\begin{equation*}
  \sup_{0\leq s\leq t}(1+s)^{\frac{3}{2}}{\mathcal{E}}_{q,l-1}(u)(s)
  \lesssim {\mathcal{E}}_{q,l}(u_0)
  +\|\mu^{l_0}u_0\|_{Z_1}^2+\|\mu^{l_0}u_0\|^2+X_{q,l-1}(u)^2(t).
\end{equation*}
This proves the decay rate stated in our Theorem for
the soft potential case, i.e., $-3<\gamma<0$ by using Strauss' Lemma.

\section*{Acknowledgments}
This work was supported by ``the Fundamental Research Funds for the Central
Universities". This work was completed when Tao Wang was visiting the Mathematical Institute
at the University of Oxford under the support of
the China Scholarship Council 201206270022.
He would like to thank Professor Gui-Qiang Chen and his group for their kind hospitality.

\end{document}